\documentclass{amsart}
\usepackage[english]{babel}
\usepackage{amsmath,amscd,amssymb,amsthm,amsxtra}
\usepackage{microtype}
\usepackage{mathrsfs}
\usepackage{mathtools}
\usepackage{accents}
\usepackage{xcolor,graphicx,esint}
\usepackage{epsfig,epstopdf}
\usepackage[margin=1.0in]{geometry}

\theoremstyle{plain}
\newtheorem{theorem}{Theorem}[section]
\newtheorem{corollary}[theorem]{Corollary}
\newtheorem{proposition}[theorem]{Proposition}
\newtheorem{lemma}[theorem]{Lemma}

\theoremstyle{definition}

\newtheorem{remark}[theorem]{Remark}

\theoremstyle{remark}
\numberwithin{theorem}{section}
\numberwithin{equation}{section}
\numberwithin{figure}{section}

\def\R{\mathbb{R}}
\def\S{\mathbb{S}}

\def\1{{\bf 1}}
\def\e{\mathrm{e}}
\def\d{\mathrm{d}}

\def\del{\delta}
\def\l{\lambda}
\def\nab{\nabla}
\def\om{\omega}
\def\Om{\Omega}
\def\vep{\varepsilon}
\def\Up{\Upsilon}

\DeclareMathOperator*{\argmax}{argmax}
\DeclareMathOperator{\co}{co}

\DeclareMathOperator{\dist}{dist}

\DeclareMathOperator{\dom}{dom}
\DeclareMathOperator{\Id}{Id}

\DeclareMathOperator*{\osc}{osc}

\DeclareMathOperator{\spt}{spt}

\def\pa{\partial}
\def\Gexp{{\mathbf G}\textnormal{-exp}}
\def\Grexp{{\mathbf G}_r\textnormal{-exp}}
\def\Hexp{{\mathbf H}\textnormal{-exp}}
\def\0{{\mathbf 0}}
\def\E{{\mathbf E}}
\def\G{{\mathbf G}}
\def\H{{\mathbf H}}
\def\S{{\mathbf S}}
\def\T{{\mathbf T}}
\def\u{{\mathbf u}}

\def\v{{\mathbf v}}
\def\V{{\mathbf V}}
\def\w{{\mathbf w}}

\def\B{\mathcal{B}}
\def\C{\mathcal{C}}
\def\El{\mathcal{E}}
\def\K{\mathcal{K}}
\def\O{\mathcal{O}}
\def\N{{\mathcal{N}}}

\begin{document}

\title[Partial Regularity in Generated Jacobian Equations]{Partial Regularity of Solutions to the Second Boundary Value Problem for Generated Jacobian Equations}

\author[Y.\ Jhaveri]{Yash Jhaveri$^\dagger$}
\thanks{$^\dagger$\,Supported in part by the ERC grant ``Regularity and Stability in Partial Differential Equations (RSPDE)''}
\address{ETH Z\"{u}rich, Department of Mathematics, R\"{a}mistrasse 101, Z\"{u}rich 8092, Switzerland}
\email{yash.jhaveri@math.ethz.ch}

\begin{abstract}
We prove that outside of a closed singular set of measure zero solutions to the second boundary value problem for generated Jacobian equations are smooth.
\end{abstract}
\maketitle
\vspace{-0.75cm}


\section{Introduction}

In this paper, we begin the development of a partial regularity theory for solutions to the second boundary value problem for a class of prescribed Jacobian equations called generated Jacobian equations.
A {\it prescribed Jacobian equation} (PJE) takes the form
\begin{equation}
\label{eqn: PJE}
\det(\nab_x[\T(x,\u(x),\nab \u(x))]) = \psi(x,\u(x),\nab \u(x))
\end{equation}
where $\T = \T(x,u,p) : \dom \T \subset \Om \times \R \times \R^n \to \R^n$ and $\psi = \psi(x,u,p) : \Om \times \R \times \R^n \to \R$; and the {\it second boundary value problem} (SBVP) asks that 
\begin{equation}
\label{eqn: boundary data}
\T_\u(\Om) = \Up
\end{equation}
for some given $\Up \subset \R^n$.
Here, $\Om$ and $\Up$ are open sets and $\T_\u(x) := \T(x,\u(x),\nab \u(x))$.
We consider the specific case when this prescription is given through the push-forward condition 
\[
(\T_\u)_\# f = g
\] 
for two probability densities $f$ and $g$ supported in $\Om$ and $\Up$ respectively.
This corresponds to
\begin{equation}
\label{eqn: RHS}
\psi(x,\u(x),\nab \u(x)) = \frac{f(x)}{g(\T(x,\u(x),\nab \u(x)))}.
\end{equation}
When the map $\T$ is generated by a function $\G : \dom \G \subset \R^n \times \R^n \times \R \to \R$, we find ourselves in the world of {\it generated Jacobian equations} (GJEs), and we call $\T_\u$ the {\it transport map associated to $\G$ and $\u$}.

PJEs, in particular, GJEs, encompass many problems in analysis, economics, and geometry (see \cite{GK} for a discussion of some of these problems as well as the many references therein).
The simplest PJE, when $\T = \T(p) = p$, is the Monge--Amp\`{e}re equation, and the SBVP corresponds to the optimal transportation problem for quadratic cost.
Here, smooth data does not ensure the existence of a smooth solution.
Rather, the problem requires an additional, strong geometric condition on the support of the target density $g$ for such a statement to hold.
Specifically, we need $\Up$ to be convex, as Caffarelli showed in \cite{C}.
(See \cite{J} for an investigation of how important, quantitatively, the convexity of $\Up$ is in guaranteeing the regularity of $\nab \u$.)
When $\T$ is highly non-linear in its variables, the complexity of the problem is compounded.
If $\T = \T(x,p)$, as it does in the optimal transport problem for general cost, not only do we need to place geometric restrictions on $\Up$, but the associated generator $\G$ of $\T$ must obey certain structural conditions to first ensure the production of $\T$ and second guarantee the regularity of solutions.
In the most complex situations, wherein $\T = \T(x,u,p)$, analogous geometric conditions on the target domain and structure conditions on $\G$ are still insufficient to yield regular solutions.
This phenomenon is exhibited, for instance, in the reflector shape design problem: Karakhanyan and Wang, in \cite{KW}, showed that smooth data may produce distinct solutions with vastly different regularity.\footnote{\,In the optimal transport problem, potentials $\u$ are unique up to the addition of a constant.
In the near-field reflector problem, however, we find that solutions may not be unique in any natural sense.}

The distinguishing feature of a general GJE from the optimal transport case is the map $\T$ may depend on the values of the potential $\u$. 
This feature will be the source of the challenges faced in this work.
The third coordinate of $\G$ in the optimal transport case for cost $c$ is a simple height parameter.
Changes in this variable translate to vertical shifts in the graph of $\G(x,y,v) = -c(x,y) - v$.
In general, changes in the third variable of an arbitrary generating function affect the shape of the graph $\G$ (see, e.g., \cite{GK,KW,T}).

As noted, the GJE setting is one in which the map $\T$ is produced from another $\G$.
In the optimal transport problem for cost $c$, the map $\T$ is generated by the equation
\[
D_x \G(x,\T) = - D_x c(x,\T) = p.
\]
Generally, the map $\T$ (along with another $\V$) is generated through the system of equations
\[
\begin{cases}
D_x\G(x,\T,\V) = p \\
\G(x,\T,\V) = u.
\end{cases}
\]
As such, $\G$ must satisfy a collection of basic structure conditions to produce $\T$.

\subsection{Structure of $\G$}
\label{sec: structure}

Our starting assumptions are three-fold: 1. $\dom \G = X \times Y \times I$ where $X, Y \subset \R^n$ are open and $I \subset \R$ is an open interval, 2. $\G$ is of class $C^{2,\alpha}_{\rm loc}(X \times Y \times I)$ for some $\alpha \in (0,1)$, and 3.
\[ 
D_v\G(x,y,v) < 0. \tag{G-Mono}
\]
Up to a change of variables, we let $I = \R$.

The remaining structure conditions on $\G$ will hold on a subset of the domain of $\G$:
\[
\mathfrak{g} := \{ (x,y,v) : v \in V_{x,y} \};
\]
for each pair $(x,y) \in X \times Y$, the set $V_{x,y}$ is some open interval (possibly empty).
We assume that $\mathfrak{g}$ is open.
In keeping with the nomenclature of \cite{GK}, the final structure conditions we impose on $\G$ are as follows.
First, we ask that the map 
\[
(y,v) \mapsto (D_x\G(x,y,v),\G(x,y,v)) \text{ is injective on } \{ (y,v) : (x,y,v) \in \mathfrak{g} \}. \tag{G-Twist}
\] 
Second, we assume that the map 
\[
x \mapsto - \frac{D_y\G (x,y,v)}{D_v\G(x,y,v)} \text{ is injective on } \{ x : (x,y,v) \in \mathfrak{g} \}. \tag{G*-Twist}
\]
Third, we suppose that
\[
\det\bigg(D_{xy}\G - D_{xv}\G \otimes \frac{D_y\G}{D_v\G}\bigg) \neq 0
\text{ on } \mathfrak{g}. \tag{G-Nondeg}
\]

We shall make some remarks on these conditions in Section~\ref{sec: prelim}.

\subsection{Statement of Main Result}

In \cite{T}, the local regularity of solutions to our SBVP is also studied. 
Under a pair of higher-order structural assumptions on the generating function $\G$ and geometric restrictions on the open, bounded sets $\Om$ and $\Up$, solutions are proved to be smooth (given smooth densities bounded away from zero and infinity in $\Om$ and $\Up$ respectively), and the transport associated to $\G$ and $\u$ is shown to be a diffeomorphism from $\Om$ onto $\Up$.
These assumptions are extensions of the MTW conditions on the cost $c$ and the $c$-convexity and $c^*$-convexity requirements on the source and target domains in the optimal transport problem for general cost (see \cite{MTW}).
We refer the reader to \cite{GK} for other results on the regularity of solutions to general generated Jacobian equations under different, but related, additional conditions on the structure of $\G$ and on the geometry of the domains of the equation.

The purpose of this paper is to show that solutions to \eqref{eqn: PJE} -- \eqref{eqn: RHS} are smooth outside a singular set of measure zero without the presence of any additional structural or geometric conditions.
Precisely, our main result is the following:

\begin{theorem}
\label{thm: main}
Let $\G$ and $\mathfrak{g}$ be as in Section~\ref{sec: structure} and $\Om \subset X$ and $\Up \subset Y$ be two open, bounded sets.
Suppose $f : \Om \to \R^+$ and $g : \Up \to \R^+$ are two continuous probability densities bounded away from zero and infinity and $\u : \Om \to \R$ is a $\G$-convex function such that $(\T_\u)_\# f = g$.
Then, for every $\beta < 1$, there exist two relatively closed sets $\mathscr{S}_\Om \subset \Om$ and $\mathscr{S}_\Up \subset \Up$ of measure zero such that $\T_\u : \Om \setminus \mathscr{S}_\Om \to \Up \setminus \mathscr{S}_\Up$ is a homeomorphism of class $C^{0,\beta}_{\rm loc}$.
If, in addition, $\G \in C^{k+2,\alpha}_{\rm loc}(\Om \times \Up \times \R)$, $f \in C^{k,\alpha}_{\rm loc}(\Om)$, and $g \in C^{k,\alpha}_{\rm loc}(\Up)$ for some $k \geq 0$ and $\alpha \in (0,1)$, then $\T_\u : \Om \setminus \mathscr{S}_\Om \to \Up \setminus \mathscr{S}_\Up$ is a diffeomorphism of class $C^{k+1,\alpha}_{\rm loc}$.
\end{theorem}

Notice that when $f$ and $g$ are just assumed to be continuous, the regular sets depend on the value of $\beta$.
In the higher regularity cases, the regular sets are independent of the values of $k$ and $\alpha$.
Recall that the SBVP for GJEs may, in general, have many solutions, all with potentially different regularity properties (again, see \cite{KW,GK}).
Yet by Theorem~\ref{thm: main}, outside sets of measure zero, all of this variety is unseen.

The proof of Theorem~\ref{thm: main} follows the global strategy of its optimal transport predecessor \cite[Theorem~1.3]{DF}.
However, new difficulties arise here coming from the additional non-linear nature of general generating functions over those that arise in optimal transportation and the non-existence of a Kantorovich formulation of the problem.
In particular, the third component $v$ of $\G$ plays no role in \cite{DF}, while its presence here is pervasive.

As far as we know, Theorem~\ref{thm: main} is the first partial regularity result on general GJEs.
That said, in the optimal transport setting, the first partial regularity result was proved by Figalli in two dimensions for quadratic cost in \cite{F}.
This two dimensional, quadratic cost result was subsequently pushed to arbitrary dimension by Figalli and Kim in \cite{FK} and then again by Goldman and Otto in \cite{GO}.
In \cite{DF}, De Philippis and Figalli extended these last quadratic cost results to general cost, while Chen and Figalli proved a partial Sobolev regularity result for general cost in \cite{CF17}.
Finally, we mention that the $\vep$-regularity techniques developed by De Philippis and Figalli, in \cite{DF}, and exploited here have been used to prove regularity results at the boundary for optimal transports in \cite{CF15} and \cite{J}.

\subsection{Organization}
This paper has four additional sections.
In Section~\ref{sec: prelim}, we introduce some more notation and some preliminary results.
Section~\ref{sec: proof of main thm} is dedicated to the proof Theorem~\ref{thm: main}.
Finally, in the last two sections, we prove the local regularity results around which the proof of our main result revolves.

\section{Preliminaries}
\label{sec: prelim}

In this section, we introduce some notation and preliminary results.
We start with some remarks on $\G$ and the structure conditions it obeys.
Then, we visit the geometry of solutions to the SVBP for GJEs.
Finally, we show that solutions to \eqref{eqn: PJE} -- \eqref{eqn: RHS} satisfy a Monge--Amp\`{e}re-type equation almost everywhere.

\subsection{Structure and Duality}
\label{sec: duality of structure conditions}

The assumption that $\mathfrak{g}$ is open is mild.
For instance, in the near-field reflector/reflector shape design problem \cite{KW}, an important model setting for the SBVP for general GJEs\,---\,wherein we have non-uniqueness of solutions and varying regularity among solutions\,---\,the set $\mathfrak{g}$ is open (see \cite[Section 3.1]{GK}).
More generally, as far as we know, the set $\mathfrak{g}$ is open in all examples of GJEs.

Thanks to (G-Mono), there exists a unique function $\H$ determined by the equation
\[
\G(x,y,\H(x,y,u)) = u,
\]
and $\H(x,y,\cdot)$ is well-defined on the (non-empty) open interval $\G(x,y,\R)$.
We call $\H$ the {\it dual} of $\G$.
In the optimal transport case, $\H(x,y,u) = -c(x,y)-u$, and we see that $\G(x,y,\R) = \R$ for all pairs $(x,y)$.
Generally, however, $\G(x,y,\R)$ maybe not be $\R$ for any pair $(x,y)$.
In addition, we define the set
\[
\mathfrak{h} := \{ (x,y,u) : u \in U_{x,y} \}
\]
with $U_{x,y} := \G(x,y,V_{x,y})$.
As $\mathfrak{g}$ is open, (G-Mono) and the continuity of $D_v\G$ together imply that the set $\mathfrak{h}$ is also open.
Hence, $\H$ is locally $C^{2,\alpha}$ on $\mathfrak{h}$ and
\[
D_u\H(x,y,u) < 0 \tag{H-Mono}.
\]
We can see that the map 
\[
(x,u) \mapsto (D_y\H(x,y,u),\H(x,y,u)) \text{ is injective on } \{ (x,u) : (x,y,u) \in \mathfrak{h} \}, \tag{H-Twist}
\] 
the map 
\[
y \mapsto - \frac{D_x\H (x,y,u)}{D_u\H(x,y,u)} \text{ is injective on } \{ y : (x,y,u) \in \mathfrak{h} \}, \tag{H*-Twist}
\]
and
\[
\det\bigg(D_{yx}\H - D_{yu}\H \otimes \frac{D_x\H}{D_u\H}\bigg) \neq 0
\text{ on } \mathfrak{h}. \tag{H-Nondeg}
\]
In particular, (G-Twist) and (H*-Twist), (G*-Twist) and (H-Twist), and (G-Nondeg) and (H-Nondeg) are respectively equivalent (see \cite[Remark 9.5]{GK} and \cite{T}).

Moreover, with $\H$, we can generate the map $\S$ and look to solve the dual generated Jacobian equation
\begin{equation}
\label{eqn: dual GJE}
\det(\nab_y[\S(y,\v(y),\nab \v(y))]) = \frac{g(y)}{f(\S(y,\v(y),\nab \v(y)))}.
\end{equation}

We will often use the following first-order identities:
\[
D_x \H = - \frac{D_x \G}{D_v \G}, \qquad D_y \H = - \frac{D_y \G}{D_v \G}, \qquad\text{and}\qquad D_u \H = \frac{1}{D_v \G}
\]
and the following second-order identities:
\[
D_{xu} \H = -\frac{D_{xv} \G + D^2_v\G D_x\H}{(D_v \G)^2}, \quad D_{yu} \H = -\frac{D_{yv} \G + D^2_v\G D_y\H}{(D_v \G)^2}, \quad\text{and}\quad D_{u}^2\H = -\frac{D_v^2 \G}{(D_v \G)^3}.
\]
Here, the derivatives of $\H$ are taken at $(x,y,u)$ and the derivatives of $\G$ at $(x,y,\H(x,y,u))$ provided, of course, $u \in \G(x,y,\R)$.
These identities are simple consequences of (G-Mono).

Now let $\E$ be the $n \times n$ matrix from (G-Nondeg):
\begin{equation}
\label{def: matrix E}
\E(x,y,v) := \bigg[ D_{xy}\G - D_{xv}\G \otimes \frac{D_y\G}{D_v\G} \bigg](x,y,v).
\end{equation}
Notice that the Jacobian determinants of the maps in (G-Twist) and (G*-Twist) are
\[
|D_v\G(x,y,v)|^n|\det (\E(x,y,v))| \qquad\text{and}\qquad |D_v\G(x,y,v)|^{-n}|\det (\E(x,y,v))|
\]
respectively.

\subsection{$\G$-convexity}
\label{sec: prelim Gconvex}

Solutions to \eqref{eqn: PJE} -- \eqref{eqn: RHS} are $\G$-convex functions.
Let us recall the definition of $\G$-convexity and some related facts, definitions, and characteristics.
We say that a function $\u : X \to \R$ is {\it $\G$-convex} if for every $x_0 \in X$, there exists a {\it focus} $(y_0,v_0) \in Y \times \R$ such that 
\[
(x_0,y_0,v_0) \in \mathfrak{g}
\]
and
\[
\u(x_0) = \G(x_0,y_0,v_0) \qquad\text{and}\qquad \u(x) \geq \G(x,y_0,v_0) \quad \forall x \in X.
\]
Notice that if $(y_0,v_0)$ and $(y_0,v_1)$ are foci for a $\G$-convex function $\u$ at the point $x_0$, then by (G-Mono), $v_0  = v_1 =  \H(x_0,y_0,\u(x_0))$.
So we can recast our definition and say that $\u : X \to \R$ is $\G$-convex if for each $x_0 \in X$, there exists a point $y_0 \in Y$ such that
\[
(x_0,y_0,\H(x_0,y_0,\u(x_0))) \in \mathfrak{g}
\]
and
\[
\u(x_0) = \G(x_0,y_0,\H(x_0,y_0,\u(x_0))) \qquad\text{and}\qquad \u(x) \geq \G(x,y_0,\H(x_0,y_0,\u(x_0))) \quad \forall x \in X.
\]

For a $\G$-convex function $\u : X \to \R$, we define its {\it $\G$-subdifferential} at $x_0$ to be the (non-empty) set
\begin{equation}
\label{def: gsub}
\pa_\G \u(x_0) := \{ y \in Y : \u(x) \geq \G(x,y,\H(x_0,y,\u(x_0))) \quad \forall x \in X \}
\end{equation}
provided
\begin{equation}
\label{eqn: gsub admiss}
(x_0,y,\H(x_0,y,\u(x_0))) \in \mathfrak{g}.
\end{equation}
Given $y \in \pa_\G \u(x_0)$, we call
\[
\mathscr{G}_{x_0,y,v}(\cdot) := \G(\cdot,y,v)
\]
with $v := \H(x_0,y,\u(x_0))$ a {\it $\G$-support} of $\u$ at $x_0$.
Recalling the {\it Fr\'{e}chet subdifferential} of a function $\u$ at $x_0$:
\[
\pa^- \u(x_0) := \{ p \in \R^n : \u(x) \geq \u(x_0) + p \cdot (x - x_0) + o(|x - x_0|) \},
\]
with $o(|x - x_0|) \to 0$ as $x \to x_0$, we see that
\begin{equation}
\label{eqn: Gsub in Fsub}
y \in \pa_\G \u(x_0) \quad\Rightarrow\quad D_x\G(x_0,y,\H(x_0,y,\u(x_0))) \in \pa^- \u(x_0).
\end{equation}
For $E \subset X$, we set
\[
\pa_\G \u(E) := \bigcup_{x \in E} \pa_\G \u(x) \qquad\text{and}\qquad  \pa^- \u(E) := \bigcup_{x \in E} \pa^- \u(x).
\]
\begin{remark}
In the optimal transport setting, the geometric condition \eqref{def: gsub} alone dictates whether or not a point $y$ is in the $\G$-subdifferential of $\u$ at $x_0$.
The admissibility condition \eqref{eqn: gsub admiss} always holds.
Yet this is not the case in general.
A simple but important consequence of this is that the $\G$-subdifferential may not be continuous in the way the $c$-subdifferential is for $c$-convex functions.
For instance, in the quadratic cost case, $c$-convexity is convexity, and given $y_k \in \pa^- \u(x_k)$ such that $y_k \to y_0$ and $x_k \to x_0$, we know that $y_0 \in \pa^- \u(x_0)$.
However, if we replace $\pa^- \u$ with $\pa_\G \u$, this implication may not hold.
\end{remark}

Akin to the Legendre transform, we define the {\it $\G$-transform} of $\u$ to be the $\H$-convex function given by
\begin{equation}
\label{def: G-transform}
\u_\G(y) := \sup_{x \in X} \H(x,y,\u(x)).
\end{equation}
In actuality, the supremum here is taken over those $x \in X$ such that $\H(x,y,\u(x))$ is defined; the $\G$-convexity of $\u$ implies that $\u(x) \in U_{x,y}$ whenever $y \in \pa_\G \u(x)$, and so the supremum is over a non-empty set.
Moreover, as noted in \cite[Section 4]{T}\footnote{\,While \eqref{eqn: Gtransforminvertible} and \eqref{eqn: Gsubdiff invertible} are mentioned in \cite[Section 4]{T}, they are not proved.
For completeness, we prove them here.
Let $y_0 \in \pa_\G \u(x_0)$ and $v_0 := \H(x_0,y_0,\u(x_0))$.
By definition, $(x_0,y_0,\u(x_0)) \in \mathfrak{h}$.
If $x_0$ is the only point at which $\H(\cdot,y_0,\u(\cdot))$ is defined, then $x_0 \in \pa_\H \u_\G(y_0)$ trivially.
On the other hand, let $x \in X$ be such that $\H(x,y_0,\u(x))$ is well-defined.
Since $\u(x) \geq \G(x,y_0,v_0)$ for all $x \in X$, we see that
\[
\H(x,y_0,\u(x)) \leq \H(x,y_0,\G(x,y_0,v_0)) = v_0 = \H(x_0,y_0,\u(x_0)).
\]
By construction, $\H(x,y_0,\cdot)$ can be evaluated at $\G(x,y_0,v_0)$.
Hence, the supremum in \eqref{def: G-transform} is achieved at $(x_0,y_0,\u(x_0))$.
That is, if  $y_0 \in \pa_\G \u(x_0)$, then $x_0 \in \pa_\H \u_\G (y_0)$ and $\u_\G(y_0) = \H(x_0,y_0,\u(x_0))$.
It then follows that
\[
\u_{\G\H}(x_0) \geq \G(x_0,y_0,\u_\G(y_0)) = \G(x_0,y_0,\H(x_0,y_0,\u(x_0))) = \u(x_0).
\]
A symmetric argument, taking $\u_\G$ in place of $\u$ and $\u_{\G\H}$ in place of $\u_\G$ above, yields that
\[
x_0 \in \pa_\H \u_\G(y_0) \quad\Rightarrow\quad y_0 \in \pa_\G \u_{\G\H}(x_0)
\]
and $\u(x_0) = \G(x_0,y_0,\H(x_0,y_0,\u(x_0))) \geq \G(x_0,y_0,\u_\G(y_0)) = \u_{\G\H}(x_0)$.
},
\begin{equation}
\label{eqn: Gtransforminvertible}
\u_{\G\H} = \u
\end{equation}
where $\v_\H(x) := \sup_{y \in Y} \G(x,y,\v(y))$ is the $\H$-transform of a given $\H$-convex function $\v : Y \to \R$, and
\begin{equation}
\label{eqn: Gsubdiff invertible}
y \in \pa_\G \u(x) \quad\Leftrightarrow\quad x \in \pa_\H \u_\G(y).
\end{equation}
For the $\H$-subdifferential, the analogue of the admissibility condition \eqref{eqn: gsub admiss} is
\[
(x,y_0,\G(x,y_0,\v(y_0))) \in \mathfrak{h},
\]
for $x \in \pa_\H \v(y_0)$.
Because we have assumed that $\G$ is of class $C^2_{\rm loc}(X \times Y \times \R)$, we find that $\G$-convex functions are locally semiconvex.
(The semiconvexity constant of $\u$ in a set depends only on the $C^0$-norm of $D^2_x\G$ in that set.)
In particular, $\G$-convex functions are locally uniformly Lipschitz and twice differentiable at almost every point (see, e.g., \cite{FF}).
This basic regularity will be the foundation of our analysis.

Since $\G$ satisfies (G-Twist) and (G-Nondeg), we can generate the maps $\Gexp_{x,u}(\cdot)$ and $\V_x(\cdot,\cdot)$ from the pair of equations
\[
\begin{cases}
D_x\G(x,\Gexp_{x,u}(p),\V_x(u,p)) = p \\
\G(x,\Gexp_{x,u}(p),\V_x(u,p)) = u 
\end{cases}
\quad\forall (p,u) \in (D_x\G,\G)(\{(x,y,v) : (x,y,v) \in \mathfrak{g}\}).
\]
Here, $D_x \G$ is evaluated at the point $(x,\Gexp_{x,u}(p),\V_x(u,p))$.
In other words,
\[
\Gexp_{x,u}(p) = y \quad\Leftrightarrow\quad p = D_x\G(x,y,\H(x,y,u))
\]
and
\[
\V_x(u,p) = \H(x,\Gexp_{x,u}(p),u)
\]
so long as $(x,y,\H(x,y,u)) \in \mathfrak{g}$.
And so \eqref{eqn: Gsub in Fsub} can be rewritten as
\begin{equation}
\label{eqn: Gsub in Gexp}
\pa_\G \u(x_0) \subset \Gexp_{x_0,\u(x_0)}(\pa^- \u(x_0)).
\end{equation}
When applying $\Gexp_{x_0,\u(x_0)}(\cdot)$ to $p \in \pa^-\u(x_0)$, we only consider those $p = D_x\G(x_0,y,\H(x_0,y,\u(x_0)))$ such that $\u(x_0) \in U_{x_0,y}$.
Hence, we see that if  $\u$ is differentiable at $x_0$, then $\pa_\G \u(x_0)$ is a singleton $\{y_0\}$ and
\begin{equation}
\label{eqn: first derivative}
\nab \u(x_0) = D_x\G(x_0,y_0,v_0),
\end{equation}
and if $\u$ is twice differentiable at $x_0$, then 
\begin{equation}
\label{eqn: second derivative}
D^2 \u(x_0) \geq D^2_x \G(x_0,y_0,v_0).
\end{equation}
In \eqref{eqn: first derivative} and \eqref{eqn: second derivative}, $v_0 := \H(x_0,y_0,\u(x_0))$.
Finally, given a $\G$-convex function $\u : X \to \R$, let us define the map (at almost every $x \in X$) $\T_\u$ by
\[
\T_\u(x) := \Gexp_{x,\u(x)}(\nab \u(x)).
\]
Even though $\T_\u$ depends on $\u$ and $\G$, we shall often suppress the second dependence for notational simplicity.

\subsection{A Monge--Amp\`{e}re-type Equation}

Set $\v := \u_\G$.
Then, $\v$ is a solution to the SBVP for \eqref{eqn: dual GJE}, the dual equation (see \cite[Lemma~4.1]{T}).
In particular, by \cite[Lemma 4.1]{T} and the remarks just before it, $(\S_\v)_\# g = f$. 
Recall that $\u$ and $\v$ are twice differentiable almost everywhere; let $\Om_1$ and $\Up_1$ be the full (Lebesgue) measure subsets of $\Om$ and $\Up$ respectively on which $\u$ and $\v$ are respectively twice differentiable.
By \eqref{eqn: Gsubdiff invertible} and \eqref{eqn: Gsub in Gexp}, we see $\T_\u$ and $\S_\v$ are inverses of one another in the sense that 
\begin{equation}
\label{eqn: transport is invertible}
\S_\v(\T_\u(x)) = x \quad\forall x  \in \Om_1 \qquad\text{and}\qquad \T_\u(\S_\v(y)) \quad\forall y \in \Up_1.
\end{equation}
Here, of course, $\S_\v(y) := \Hexp_{y,\v(y)} (\nab \v(y))$. 
Since $(\T_\u)_\# f = g$, we can apply \cite[Theorem~11.1]{V} to deduce that
\[
|\det (\nab \T_\u(x))| = \frac{f(x)}{g(\T_\u (x))} \quad \forall x \in \Om_1.
\]
Then, \eqref{eqn: first derivative} and \eqref{eqn: second derivative} imply that
\begin{equation}
\label{eqn: pointwise jacobian eqn}
\begin{split}
\det(D^2\u(x) - &D^2_x\G(x,\T_\u(x),\H(x,\T_\u(x),\u(x))) \\
&= |\det (\E(x,\T_\u(x),\H(x,\T_\u(x),\u(x))))|\frac{f(x)}{g(\T_\u(x))} \quad\text{a.e.}
\end{split}
\end{equation}
Since by (G-Nondeg), $\E$ has non-zero determinant, the nondegeneracy of the right-hand side of our Monge--Amp\`{e}re-type equation \eqref{eqn: pointwise jacobian eqn} is preserved.
(See \cite{T} for more details.)
In conclusion, a solution $\u$ to \eqref{eqn: PJE} -- \eqref{eqn: RHS} satisfies a Monge--Amp\`{e}re-type equation almost everywhere.


\section{Proof of Theorem~\ref{thm: main}}
\label{sec: proof of main thm}

Set $\v$ to be the $\G$-transform of $\u$ and let $\Om_1$, $\Up_1$, and $\S_\v$ be as in \eqref{eqn: transport is invertible}.
Consider the set
\[
\Om_2 := \Om_1 \cap \T_\u^{-1} (\Up_1) \subset \Om,
\]
and observe that $|\Om \setminus \Om_2| = 0$ since $(\T_\u)_\# f = g$ and the densities $f$ and $g$ are bounded away from zero and infinity.
Recall that $(\S_\v)_\# g = f$.

Fix $x' \in \Om_2$.
Since $x'$ is a point of differentiability for $\u$, the $\G$-subdifferential of $\u$ at $x'$ is a singleton (see \eqref{eqn: Gsub in Gexp}): $\pa_\G \u(x') = \{ \T_\u(x') \}$.
Set $y' := \T_\u(x')$ and $v' := \H(x',\T_\u(x'), \u(x')))$.
Note that $y' \in \Up_1$.
Up to a translation, we can assume that $(x',y',v') = (0,0,0)$.
Furthermore, up to subtracting $\G(\cdot,0,0)$, we can assume that $\u(0) = 0$ and its $\G$-support at $(0,0,0)$ is identically zero; that is, $\mathscr{G}_{0,0,0}(x) = 0$ and $\u(x) \geq 0$ for all $x \in \Om$.
In turn, 
\begin{equation}
\label{eqn: kill x der of G}
D_x \G(\cdot,0,0) = 0.
\end{equation}
With these normalizations in hand, define
\[
\hat{\G}(x,y,v) := \G(x,y,v + \H(0,y,0)) \qquad\text{and}\qquad \hat{\u}(x) := \u(x).
\]
Notice that $\H(0,y,0)$ may not be defined at all points $y \in Y$ or even all points $y \in \Up$.
However, it is well-defined in $B_\vep \subset \Up$ for some $\vep > 0$ by the implicit function theorem and (G-Mono).
Furthermore, as $\u$ is twice differentiable at the origin and $\nab\u(0) = 0$, \cite[Theorem~14.25]{V} implies that
\[
\pa^-\u(x) = D^2\u(0)x + o(|x|).
\]
So using \eqref{eqn: Gsub in Gexp}, we can find an $\epsilon > 0$ such that
\begin{equation}
\label{eqn: localization 1}
\pa_\G \u(B_\epsilon) \subset B_\vep \subset \Up.
\end{equation}
Here, $\vep$ and $\epsilon$ are not necessarily equal.
For each $(x,y) \in B_\epsilon \times B_\vep$, let $\hat{V}_{x,y} := V_{x,y} - \H(0,y,0)$.
Define
\[
\hat{\mathfrak{g}} := \{ (x,y,v) : (x,y) \in B_\epsilon \times B_\vep \text{ and } v \in \hat{V}_{x,y} \}.
\]
Observe that $\hat{\G}$ satisfies (G-Twist), (G*-Twist), and (G-Nondeg) on $\hat{\mathfrak{g}}$ (also (G-Mono)).
In $B_\epsilon$, we find that $\hat{\u}$ is $\hat{\G}$-convex.
Indeed, let $x_0 \in B_\epsilon$, $y_0 \in \pa_\G \u(x_0)$, and $\hat{v}_0 := \H(x_0,y_0,\u(x_0)) - \H(0,y_0,0)$.
Then, $(x_0,y_0,\hat{v}_0) \in \hat{\mathfrak{g}}$ and
\[
\hat{\G}(x,y_0,\hat{v}_0) = \G(x,y_0,\H(x_0,y_0,\u(x_0))) \leq \u(x) = \hat{\u}(x) \quad \forall x \in B_\epsilon
\]
with equality at $x = x_0$.
In particular,
\begin{equation}
\label{eqn: subdiff matching}
\pa_{\hat{\G}} \hat{\u}(x) = \pa_\G \u(x) \quad\forall x \in B_\epsilon.
\end{equation}
Setting
\begin{equation}
\label{eqn: hatf and hatg}
\hat{f} := f\1_{B_\epsilon} \qquad\text{and}\qquad \hat{g} := g\1_{\pa_{\hat{\G}} \hat{\u}(B_\epsilon)},
\end{equation}
we claim that 
\begin{equation}
\label{eqn: first pushforward}
(\T_{\hat{\u}})_\# \hat{f} = \hat{g}.
\end{equation}
Note that the dual of $\hat{\G}$ is
\[
\hat{\H}(x,y,u) := \H(x,y,u) - \H(0,y,0).
\]
And so using \eqref{eqn: subdiff matching} and recalling \eqref{eqn: Gsub in Gexp}, we see that $\T_\u|_{B_\epsilon} = \T_{\hat{\u}}$.
Thus, recalling \eqref{eqn: subdiff matching}, it suffices to show that $\T_{\u}^{-1}(\pa_{\G} \u(B_\epsilon)) \setminus B_\epsilon$ has measure zero.
To this end, observe that if $x \in \T_{\u}^{-1}(\pa_{\G} \u(B_\epsilon)) \setminus B_\epsilon$, then there exists an $x_\epsilon \in B_\epsilon$ such that $\pa_\G \u(x) \cap \pa_\G \u(x_\epsilon)$ is non-empty.
Therefore, as $x \neq x_\epsilon$ and recalling \eqref{eqn: Gsubdiff invertible}, we see that
\[
\T_{\u}^{-1}(\pa_{\G} \u(B_\epsilon)) \setminus B_\epsilon \subset \T_{\u}^{-1}(\{ \text{non-differentiability points of } \v\}).
\]
Since 
\[
|\T_{\u}^{-1}(\{ \text{non-differentiability points of }\v\})| = 0,
\] 
because $\v$ is semiconvex and $f$ is bounded away from zero, the claim holds.

By construction, $\hat{\G}$ is such that
\[
\hat{\G}(\cdot,0,0) = \hat{\G}(0,\cdot,0) \equiv 0,
\]
and $\hat{\u}$ is such that
\[
\hat{\u}(0) = \nab \u(0) = 0.
\]
Thus, Taylor expanding $\hat{\G}$ and $\hat{\u}$ around the origin yields
\[
\begin{split}
\hat{\G}(x,y,v) =  D_v &\hat{\G}(0,0,0)v + D_{xy} \hat{\G}(0,0,0)x \cdot y \\
&+ D_{xv} \hat{\G}(0,0,0)\cdot vx + D_{yv} \hat{\G}(0,0,0)\cdot vy + \frac{1}{2} D^2_v \hat{\G}(0,0,0)v^2 + O( |x|^{2+\alpha} + |y|^{2+\alpha} + |v|^{2+\alpha})
\end{split}
\]
and
\[
\hat{\u}(x) = \frac{1}{2} D^2 \hat{\u}(0) x \cdot x + o(|x|^2).
\]
Let 
\[
a := - D_v \hat{\G}(0,0,0) > 0,\qquad M := D_{xy}\hat{\G}(0,0,0),\qquad\text{and}\qquad P := D^2\hat{\u}(0).
\]
As $M = \E(0,0,0)$, with $\E$ defined in \eqref{def: matrix E}, $\det(M) \neq 0$.
Hence, using \eqref{eqn: pointwise jacobian eqn} and \eqref{eqn: kill x der of G}, we find that $\det(P) = \det(D^2\u(0)) > 0$; that is, $P$ is positive definite and symmetric.
Therefore, after the change of coordinates 
\[
(x,y,v) \mapsto (\tilde{x}, \tilde{y}, \tilde{v}) := (P^{1/2}x, P^{-1/2}M^t y, av),
\] 
we see that
\begin{equation}
\label{eqn: Gtalyor after localization and cov}
\tilde{\G}(\tilde{x}, \tilde{y}, \tilde{v}) := \hat{\G}(x,y,v) =  - \tilde{v} + \tilde{x} \cdot \tilde{y} + b_1 \cdot \tilde{v}\tilde{x} + b_2 \cdot \tilde{v}\tilde{y} +  c_3 \tilde{v}^2 + O(|\tilde{x}|^{2+\alpha} + |\tilde{y}|^{2+\alpha} + |\tilde{v}|^{2+\alpha}),
\end{equation}
with
\[
b_1 := \frac{1}{a}D_{xv} \hat{\G}(0,0,0)P^{-1/2}, \quad b_2 := \frac{1}{a}D_{yv} \hat{\G}(0,0,0)[M^t]^{-1}P^{1/2}, \quad\text{and}\quad c_3 := \frac{1}{2a^2} D^2_v \hat{\G}(0,0,0).
\]
Also,
\begin{equation}
\label{eqn: utalyor after localization and cov}
\tilde{\u}(\tilde{x}) :=  \hat{\u}(x) =  \frac{1}{2}|\tilde{x}|^2 + o(|\tilde{x}|^2)
\end{equation}
and is $\tilde{\G}$-convex in $P^{1/2}B_\epsilon$.
In particular,
\[
\pa_{\tilde{\G}} \tilde{\u}(\tilde{x}) = P^{-1/2}M^t\pa_{\hat{\G}} \hat{\u}(x) \qquad\text{with}\qquad  x = P^{-1/2}\tilde{x}.
\]
Now admissibility is with respect to 
\[
\tilde{\mathfrak{g}} := \{ (\tilde{x}, \tilde{y}, \tilde{v}) : (\tilde{x}, \tilde{y}) \in P^{1/2}B_\epsilon \times P^{-1/2}M^tB_\vep \text{ and } \tilde{v} \in a\hat{V}_{P^{-1/2}\tilde{x},[M^t]^{-1}P^{1/2}\tilde{y}} \}.
\]

Additionally, letting $\tilde{\H}$ be the dual of $\tilde{\G}$, we see that
\begin{equation}
\label{eqn: Htalyor after localization and cov}
\tilde{\H}(\tilde{x}, \tilde{y}, \tilde{u}) =  - \tilde{u} + \tilde{x} \cdot \tilde{y} - b_1 \cdot \tilde{u}\tilde{x} - b_2 \cdot \tilde{u}\tilde{y} -  c_3 \tilde{u}^2 + O( |\tilde{x}|^{2+\alpha} + |\tilde{y}|^{2+\alpha} + |\tilde{u}|^{2+\alpha}).
\end{equation}
By construction,
\[
\tilde{\H}(\cdot,0,0) = \tilde{\H}(0,\cdot,0) \equiv 0.
\]
Here, $\tilde{u} := u$.
Furthermore, if we set
\[
\tilde{f}(\tilde{x}) := \det(P^{-1/2})\hat{f}(P^{-1/2}\tilde{x}) \qquad\text{and}\qquad \tilde{g}(\tilde{y}) := |\det([M^t]^{-1}P^{1/2})|\hat{g}([M^t]^{-1}P^{1/2} \tilde{y}),
\]
then from \eqref{eqn: hatf and hatg}, \eqref{eqn: pointwise jacobian eqn}, and \eqref{eqn: kill x der of G}, we deduce that
\begin{equation}
\label{eqn: f0 equals g0}
\frac{\tilde{f}(0)}{\tilde{g}(0)} = \frac{\det(P^{-1/2})}{|\det([M^t]^{-1}P^{1/2})|}\frac{f(0)}{g(0)} = \frac{\det(P^{-1})}{|\det(M^{-1})|}  \frac{\det(D^2\u(0))}{|\det(\E(0,0,0))|} = 1.
\end{equation}
Moreover, $(\T_{\tilde{\u}})_\#\tilde{f} = \tilde{g}$ by \eqref{eqn: first pushforward} and construction.

Consider the rescalings\footnote{\ The most basic generating function, coming from the optimal transport problem with quadratic cost, is invariant under parabolically quadratic rescalings, thinking of $v$ as time: set
\[
\G_r(x,y,v) := \frac{\G(rx, ry, r^2v)}{r^{2}};
\] 
if $\G(x,y,v) = x \cdot y - v$, then $\G_r = \G$.
Rescaling in this way suggests that $\hat{\G}$ need not account for $D_{xv}\G(0,0,0)$, $D_{yv}\G(0,0,0)$, or $D^2_v\G(0,0,0)$ being non-zero, a heuristic confirmation of our choice for $\hat{\G}$.}
\[
\G_r(\tilde{x},\tilde{y},\tilde{v}) := \frac{\tilde{\G}(r\tilde{x},r\tilde{y},r^2 \tilde{v})}{r^{2}}, \qquad \H_r(\tilde{x},\tilde{y},\tilde{v}) := \frac{\tilde{\H}(r\tilde{x},r\tilde{y},r^2 \tilde{u})}{r^{2}}, \qquad\text{and}\qquad \u_r(\tilde{x}) := \frac{\tilde{\u}(r\tilde{x})}{r^2}.
\]
Since $\mathfrak{g}$ and $\mathfrak{h}$ are open and $\G$ is of class $C^2_{\rm loc}$, we have that $B_8 \times B_8 \times (-64,64) \subset \mathfrak{g}_r, \mathfrak{h}_r$ for all $r$ sufficiently small.
In addition to the openness of $\mathfrak{g}$, using \eqref{eqn: Gtalyor after localization and cov} and the $C^2$ regularity of $\G$, we can ensure that $B_8 \times (-64,64) \times B_8 \subset \dom \Grexp$ for all $r$ small enough.\footnote{\,By the $C^2$ regularity of $\G$, the openness of $\mathfrak{g}$, and the inverse function theorem (recall (G-Twist)), we find that the family functions $\{ {\mathbf F}_x := (\Gexp_x,\V_x)\}_{x \in \mathcal{O}}$ varies in a $C^1$ fashion (in $x$) in some open set $\mathcal{O}$.
Hence, by continuity and the openness of $\mathfrak{g}$, we can find an open subset of the origin in $\R^{2n+1}$ on which $\Gexp$ is well-defined.
The rescaling $(x,u,p) \mapsto (x/r,u/r^2,p/r)$, therefore, permits this inclusion.}
Here, $\mathfrak{g}_r$ is the set on which $\G_r$ satisfies (G-Twist), (G*-Twist), and (G-Nondeg); and $\H_r$ and $\mathfrak{h}_r$ are as expected.
Also, define
\[
f_r(\tilde{x}) := \frac{\tilde{f}(r\tilde{x})}{\tilde{f}(0)} \qquad\text{and}\qquad g_r(\tilde{y}) := \frac{\tilde{g}(r\tilde{y})}{\tilde{g}(0)}.
\]
Since $f$ and $g$ are bounded away from zero and infinity, \eqref{eqn: f0 equals g0} implies that $\tilde{f}(0) = \tilde{g}(0)$.
Therefore, from the continuity of $f$ and $g$, we deduce that
\begin{equation*}
\|f_r - 1\|_{L^\infty(B_4)} + \|g_r - 1\|_{L^\infty(B_4)} \leq \del
\end{equation*}
with $\del = \del(r) \to 0$ as $r \to 0$.
Using the push-forward condition $(\T_{\tilde{\u}})_\#\tilde{f} = \tilde{g}$, we find that $(\T_{\u_r})_\# f_r = g_r$.
Moreover, from \eqref{eqn: Gtalyor after localization and cov} and \eqref{eqn: Htalyor after localization and cov}, we determine that
\begin{equation}
\label{eqn: rescaled G small}
\|\G_r - \tilde{x}\cdot\tilde{y} + \tilde{v}\|_{C^{2,\alpha}(B_8 \times B_8 \times (-64,64))} + \|\H_r - \tilde{x}\cdot\tilde{y} + \tilde{u}\|_{C^{2,\alpha}(B_8 \times B_8 \times (-64,64))} \leq \del
\end{equation}
where $\del = \del(r) \to 0$ as $r \to 0$.
Furthermore, from \eqref{eqn: utalyor after localization and cov}, we see that
\begin{equation*}
\bigg{\|}\u_r - \frac{1}{2}|\tilde{x}|^2\bigg{\|}_{C^0(B_4)} \leq \eta
\end{equation*}
where $\eta = \eta(r) \to 0$ as $r \to 0$.

As we proved \eqref{eqn: localization 1},
we deduce that
\begin{equation}
\label{eqn: rescaled subdiff small}
\pa^-\u_r(\tilde{x}) \subset B_\varrho(\tilde{x}) \quad\forall \tilde{x} \in B_3
\end{equation}
where $\varrho = \varrho(r) \to 0$ as $r \to 0$.
Additionally, from \eqref{eqn: rescaled G small}, we find that 
\[
\|\Grexp_{\tilde{x},\tilde{u}} (\tilde{p}) - \tilde{p}\|_{C^1(B_8 \times (-64,64) \times B_8)} \leq \del.
\]
Therefore, using \eqref{eqn: Gsub in Gexp} and \eqref{eqn: rescaled subdiff small}, we obtain that
\begin{equation}
\label{eqn: G-r sub in ball}
\pa_{\G_r}\u_r(\tilde{x}) \subset B_{\rho}(\tilde{x}) \quad \forall\tilde{x} \in B_2
\end{equation}
where $\rho = \rho(r) \to 0$  as $r \to 0$.
Since $\hat{\u}_{\hat{\G}}$ is twice differentiable at $0$, the $\G_r$-transform of $\u_r$ is also twice differentiable at $0$; let $\v_r$ be the $\G_r$-transform of $\u_r$.
In addition, $D^2\v_r(0) = \Id$\footnote{\ We can see this by differentiating the equations 
\begin{align*}
\nab \u_r(\tilde{x}) = D_{\tilde{x}}\G_r(\tilde{x},\T_{\u_r}(\tilde{x}),\H_r(\tilde{x},\T_{\u_r}(\tilde{x}),\u_r(\tilde{x}))) \qquad\text{and}\qquad
\nab \v_r(\tilde{y}) = D_{\tilde{y}}\H_r(\S_{\v_r}(\tilde{y}),\tilde{y},\G_r(\S_{\v_r}(\tilde{y}),\tilde{y},\v_r(\tilde{y})))
\end{align*}
at $0$ and using \eqref{eqn: Gtalyor after localization and cov}, \eqref{eqn: Htalyor after localization and cov}, and that $[\nab \T_{\u_r}(0)]^{-1} = \nab \S_{\v_r}(0)$.}.
So arguing as we did to prove \eqref{eqn: G-r sub in ball}, we find that
\begin{equation}
\label{eqn: H-r sub in ball}
\pa_{\H_r}\v_r(\tilde{y}) \subset B_{\rho^*}(\tilde{y}) \quad \forall\tilde{y} \in B_2
\end{equation}
where $\rho^* = \rho^*(r) \to 0$  as $r \to 0$.
Hence, if
\begin{equation*}
\C := \overline{B}_1 \qquad\text{and}\qquad  \K := \pa_{\G_r}\u_r(\overline{B}_1)
\end{equation*}
and $r$ is sufficiently small, then we can force 
\begin{equation}
\label{eqn: C2 inclusion}
B_{1/2} \subset \K \subset B_2
\end{equation}
thanks to \eqref{eqn: G-r sub in ball}, \eqref{eqn: H-r sub in ball}, and duality.
Observe that $\C$ is convex by construction and $\K$ is closed being the $\G_r$-subdifferential of a compact set and recalling \eqref{eqn: C2 inclusion}, \eqref{eqn: rescaled G small}, and the inclusion $B_8 \times B_8 \times (-64,64) \subset \mathfrak{g}_r$.
Finally, recalling \eqref{eqn: f0 equals g0} and arguing as we did to prove \eqref{eqn: first pushforward}, we have that
\[
(\T_{\u_r})_\# (f_r\1_{\C}) = g_r\1_{\K}.
\]

The remainder of the proof of Theorem~\ref{thm: main} is identical to the optimal transport case after replacing \cite[Theorem~4.3]{DF}, \cite[Theorem~5.3]{DF}, and \cite[Corollary~4.6]{DF} by Theorem~\ref{thm: close implies C1b}, Theorem~\ref{thm: hr}, and Corollary~\ref{cor: Tu is locally open} respectively.
We refer the reader to \cite{DF} rather than including the details.


\section{$C^{1,\beta}$-regularity and Strict $\G$-convexity}
In this section, we prove an $\vep$-regularity result and exhibit some of its consequences.
Before stating it, let us introduce some notation.
Set
\[
\B_R := B_{2R} \times B_{2R} \times (-R^3,R^3) \subset \R^{2n+1}. 
\]
Furthermore, we say a set $E$ is $C$-semiconvex if every point on the boundary of $E$ can be touched from outside by a ball of radius $1/C$.
 
\begin{theorem}
\label{thm: close implies C1b}
Let $\C$ be a closed, $(C_2\delta)$-semiconvex set and $\K$ be a closed set such that
\begin{equation}
\label{eqn: thm bounds on Ci}
B_{1/2} \subset \C, \K \subset B_2,
\end{equation}
$f$ and $g$ be two densities supported on $\C$ and $\K$ respectively, and $\u$ be a $\G$-convex function such that $\pa_\G \u(\C) \subset B_2$ and $(\T_\u)_\# f = g$.
In addition, suppose that $\B_4\subset \mathfrak{g}, \mathfrak{h}$.
For every $\beta \in (0,1)$, there exist constants $\del_0, \eta_0 \in (0,1)$ such that the follow holds: if
\begin{equation}
\label{eqn: thm densities close to constant}
\|f - \1_\C\|_{L^\infty(B_4)} + \|g - \1_\K\|_{L^\infty(B_4)} \leq \del_0,
\end{equation}
\begin{equation}
\label{eqn: thm G close to linear}
\|\G - x \cdot y + v\|_{C^{2,\alpha}(\B_4)} + \|\H - x \cdot y + u\|_{C^{2,\alpha}(\B_4)} \leq \del_0,
\end{equation}
and
\begin{equation}
\label{eqn: thm u close to parabola}
\bigg{\|}\u - \frac{1}{2}|x|^2\bigg{\|}_{C^0(B_4)} \leq \eta_0,
\end{equation}
then $\u \in C^{1,\beta}(B_{1/6})$.
\end{theorem}

Theorem~\ref{thm: close implies C1b} will follow from its pointwise version Proposition~\ref{prop: C1b ptwise}.

\begin{proposition}
\label{prop: C1b ptwise}
Let $\C$ be a closed, $(C_3\delta)$-semiconvex set and $\K$ be a closed set such that
\begin{equation*}
B_{1/3} \subset \C, \K \subset B_3,
\end{equation*}
$f$ and $g$ be two densities supported on $\C$ and $\K$ respectively, and $\u$ be a $\G$-convex function such that $\u(0) = 0$, $\pa_{\G} \u(\C) \subset B_3$, and $(\T_{\u})_\# f = g$.
In addition, suppose that $\B_3 \subset \mathfrak{g}, \mathfrak{h}$,
\begin{equation}
\label{eqn: extra 1}
\G(\cdot,0,0) = \G(0,\cdot,0) = \H(\cdot,0,0) = \H(0,\cdot,0) \equiv 0,
\end{equation}
and
\begin{equation}
\label{eqn: extra 2}
D_{v}\G(0,0,0) = D_{u}\H(0,0,0) = -1
\qquad\text{and}\qquad
D_{xy} \G(0,0,0) = D_{xy} \H(0,0,0) = \Id.
\end{equation}
For every $\beta \in (0,1)$, there exist constants $\del, \eta \in (0,1)$ such that the follow holds: if
\begin{align}
\label{eqn: prop densities close to constant}
\|f - 1\|_{L^\infty(\C)} &+ \|g - 1\|_{L^\infty(\K)} \leq \del,
\\
\label{eqn: Gbar small}
\|\G - x \cdot y + v\|_{C^{2,\alpha}(\B_3)} &+ \|\H - x \cdot y + u\|_{C^{2,\alpha}(\B_3)}\leq \del,
\end{align}
and
\begin{equation}
\label{eqn: prop u close to parabola}
\bigg{\|}\u - \frac{1}{2}|x|^2\bigg{\|}_{C^0(B_3)} \leq \eta,
\end{equation}
then $\u \in C^{1,\beta}(0)$.
\end{proposition}

The proof of Proposition~\ref{prop: C1b ptwise} makes use of two lemmas.
The first is a compactness result that allows us to approximate $\u$ with a solution to an optimal transport problem with quadratic cost.
The second is an estimate on the $\G$-subdifferential of $\u$ in terms of the gradient map of the convex potential that approximates $\u$ found in the first lemma.

\begin{lemma}
\label{lem: u is close to ot soln}
Let $\C$ be a closed, $(C_R\delta)$-semiconvex set and $\K$ be a closed set such that
\begin{equation}
\label{eqn: bounds on Ci}
B_{1/R} \subset \C, \K \subset B_R
\end{equation}
for some $R \geq 3$, $f$ and $g$ be two densities supported on $\C$ and $\K$ respectively, and $\u : B_R \to (-R^2,R^2)$ be a $\G$-convex function such that $\pa_\G \u(\C) \subset B_R$ and $(\T_\u)_\# f = g$.
In addition, suppose that $\B_{R} \times \subset \mathfrak{g}$.
Also, let $\rho > 0$ be such that $|\C| = |\rho \K|$ and $\w$ be a convex function such that $(\nab \w)_\# \1_{\C} = \1_{\rho \K}$ with $\w(0) = \u(0)$.
Then, there exists an increasing function $\om : \R^+ \to \R^+$, depending only on $R$, satisfying $\om(\del) \geq \del$ and $\om(0^+) = 0$ such that if
\begin{equation}
\label{eqn: densities close to constant}
\|f - 1\|_{L^\infty(\C)} + \|g - 1\|_{L^\infty(\K)} \leq \del
\end{equation}
and
\begin{equation}
\label{eqn: G close to linear}
\|\G - x \cdot y + v\|_{C^2(\B_{R})} + \|\H - x \cdot y + u\|_{C^2(\B_{R})} \leq \del,
\end{equation}
then
\begin{equation}
\label{eqn: u close to ot potential}
\|\u - \w\|_{C^0(B_{1/R})} \leq \om(\del).
\end{equation}
\end{lemma}

\begin{proof}
Suppose, to the contrary, that the lemma is false.
Then, there exists an $\vep_0 > 0$ and sequences of closed sets $\C_j$ and $\K_j$ satisfying \eqref{eqn: bounds on Ci} with $\C_j$ being $(C_R/j)$-semiconvex, functions $f_j$ and $g_j$ satisfying \eqref{eqn: densities close to constant} with $\del = 1/j$, and generating functions $\G_j$ satisfying \eqref{eqn: G close to linear} also with $\del = 1/j$ such that
\begin{equation}
\label{eqn: contradications j}
\u_j(0) = \w_j(0) = 0 \qquad\text{and}\qquad \|\u_j - \w_j\|_{C^0(B_{1/R})} \geq \vep_0
\end{equation}
where $\u_j$ and $\w_j$ are as in the statement of the lemma.
Moreover, 
$\B_{R}\subset \mathfrak{g}_j$.

Using the push-forward condition $(\T_{\u_j})_\# f_j = g_j$, \eqref{eqn: bounds on Ci}, and \eqref{eqn: densities close to constant}, we find that
\begin{equation}
\label{eqn: rho to 1}
\rho_j = (|\C_j|/|\K_j|)^{1/n} \to 1
\end{equation}
as $j \to \infty$.
Now let us extend each $\w_j|_{\C_j}$ to a convex function on $B_R$, setting
\[
\w_j(x) := \sup_{z \in \C_j, \, p \in \pa^-\w_j(z)} \w_j(z) + p \cdot (x - z).
\]
From \eqref{eqn: rho to 1}, we see that
\begin{equation*}
\pa^- \w_j(B_R) \subset B_{{\rho_j} R} \subset B_{2R}
\end{equation*}
for $j \gg 1$.
Hence, the family $\w_j$ is uniformly Lipschitz (recall the equality in \eqref{eqn: contradications j}), and so, up to a subsequence, $\w_j$ converges uniformly in $B_R$ to some convex function $\w_\infty$.
Similarly, let us extend $\u_j|_{\C_j}$ to $B_R$:
\[
\u_j(x) := \sup_{z \in \C_j, \, y \in \pa_{\G_j}\u_j(z)} \G_j(x,y,\H_j(z,y,\u_j(z))).
\]
Given $x_0 \in B_R$, let $(z_0,y_0)$ be a pair at which the above supremum is attained.
Then, using \eqref{eqn: G close to linear}, we see that $v_0 := \H_j(z_0,y_0,\u_j(z_0)) \in (-R^3,R^3)$, and it follows that $(x_0,y_0,v_0) \in \mathfrak{g}_j$.
Consequently, these extensions are $\G_j$-convex in $B_R$. 
So from \eqref{eqn: G close to linear}, in particular, since the $C^1$-norms of $\G_j$ are uniformly bounded, and as $\pa_{\G_j} \u_j(B_R) \subset B_R$, taking $j \gg 1$, we determine that the collection $\u_j$ is uniformly $(R+1)$-Lipschitz (again, recall that $\u_j(0) = 0$).
Thus, up to a subsequence, $\u_j$ converges uniformly in $B_R$ to some convex function $\u_\infty$.
Moreover, by \eqref{eqn: contradications j},
\begin{equation}
\label{eqn: contradications infty}
\u_\infty(0) = \w_\infty(0) = 0 \qquad\text{and}\qquad \|\u_\infty - \w_\infty\|_{C^0(B_{1/R})} \geq \vep_0.
\end{equation}
Up to subsequences, the sets $\C_j$ converge in the Hausdorff sense to some 
\[
B_{1/R} \subset \C_\infty \subset B_R.
\]
Also, since each $\C_j$ is $(C_R/j)$-semiconvex, it follows that $\C_\infty$ is convex and
\[
\dist(\pa \C_j, \pa \C_\infty)+\dist(\pa \C_\infty, \pa \C_j) \to 0.
\] 
Then, arguing exactly as is \cite[Theorem]{B}, we find that $|\C_j \Delta \C_\infty| \to 0$, or, equivalently, that $\1_{\C_j}$ converges in $L^1$ to $\1_{\C_\infty}$.
Furthermore, using \eqref{eqn: densities close to constant} and as $\rho_j \to 1$, we see that $f_j$ and $\1_{\rho_j\C_j}$ converge in $L^1$ to $\1_{\C_\infty}$.
In addition, up to subsequences, $g_j$ and $\1_{\rho_j\K_j}$ converge weakly-* in $L^\infty$ to a density $g_\infty$.

By \cite[Theorem~5.20]{V} and the uniqueness of optimal transports, we see that $\nab \w_\infty$ is the optimal transport for the quadratic cost $-x \cdot y$ taking $\1_{\C_\infty}$ to $g_\infty$.
If $\nab \u_\infty = \nab \w_\infty$ almost everywhere in $\C_\infty$, then the equality in \eqref{eqn: contradications infty} implies that $\u_\infty = \w_\infty$, contradicting the inequality in \eqref{eqn: contradications infty}.
It then follows that there exists an increasing function $\om_R : \R^+ \to \R^+$, depending only on $R$, such that $\om_R(0^+) = 0$ and \eqref{eqn: u close to ot potential} holds.
Taking $\om(\del) := \max\{ \om_R(\del),\del\}$ concludes the proof.

Define $\pi_j := (\Id,\T_{\u_j})_\#f_j$.
By construction, this family of measures is tight and
\[
\spt \pi_j \subset \bigcup_{x \in \C_j} \{ (x,y) : y \in \pa_{\G_j} \u_j(x) \}.
\]
So $\pi_j$ converges weakly to some measure $\pi_\infty$ whose marginals are $\1_{\C_\infty}$ and $g_\infty$.
Furthermore, for any $\{ (x_k.y_k) \}_{k = 1}^N \subset \spt \pi_\infty$, there exist sequences $\{(x_k^j,y_k^j)\}_{k = 1}^N \subset \spt \pi_j$ such that $(x_k^j,y_k^j) \to (x_k,y_k)$ for each $k = 1, \dots, N$ and
\[
\sum_{k=1}^N \u_j(x_{k+1}^j) \geq \sum \G_j(x_{k+1}^j,y_k^j,\H_j(x_k^j,y_k^j,\u_j(x_k^j)))
\]
with $(x_{N+1}^j,y_{N+1}^j) = (x_1^j,y_1^j)$.
This is just the $\G_j$-convexity of $\u_j$.
Recalling \eqref{eqn: G close to linear} and that $\u_j$ converges uniformly to $\u_\infty$, taking the limit as $j \to \infty$, we deduce that
\[
0 \geq \sum_{k=1}^N (x_{k+1}-x_k) \cdot y_k.
\]
In other words, the support of $\pi_\infty$ is $c$-cyclically monotone for the quadratic cost.
Therefore,
\[
\pi_\infty = (\Id,\nab \u_0)_\#\1_{\C_\infty},
\]
and $\nab \u_0$ is the optimal transport for the quadratic cost $-x \cdot y$ taking $\1_{\C_\infty}$ to $g_\infty$.
In particular, $\nab \u_0 = \nab \w_\infty$ almost everywhere in $\C_\infty$.
Using that $f_j$ converges in $L^1$ to $\1_{\C_\infty}$ and arguing as in the proof of \cite[Corollary~5.23]{V}, we see that $\T_{\u_j}$ converges to $\nab \u_0$ in measure in $B_R$ and as distributions, up to a further subsequence.
By \eqref{eqn: G close to linear} and since $\u_j$ converges to $\u_\infty$ uniformly, we have that $\T_{\u_j}$ also converges to $\nab \u_\infty$ as distributions in $B_R$.
Hence, by the local integrability of $\nab \u_0$ and $\nab \u_\infty$, we determine that $\nab \u_\infty = \nab \u_0$ almost everywhere in $\C_\infty$, as desired.
\end{proof}

\begin{remark}[A remark on the regularity of $\C$ and the proof of Lemma~\ref{lem: u is close to ot soln}]
In the optimal transportation setting, De Philippis and Figalli appeal to the strong stability results available for solutions, thanks to the Kantorovich formulation of the problem, to show that $\nab \u_\infty = \nab \w_\infty$ almost everywhere.
Here, however, the set $\C$ has to have some regularity to deduce $L^1$ convergence of the contradiction sequence's source densities and, in turn, prove the same equality.
An inspection of the proof of \cite[Theorem]{B} shows that $|\C_j \Delta \C_\infty|$ goes to zero, i.e., $\1_{\C_j}$ converges in $L^1$ to $\1_{\C_\infty}$, provided that the boundaries $\pa \C_j$ converge uniformly to the boundary $\pa \C_\infty$ and have ($n$-dimensional Lebesgue) measure zero.
Therefore, Lemma~\ref{lem: u is close to ot soln} can be applied, by Arzel\`a-Ascoli, if $\C$ is a Lipschitz set whose boundary's Lipschitz constant depends only on $R$, for example.
So the regularity assumption on $\C$ in Theorem~\ref{thm: close implies C1b} and the following lemmas, propositions, and theorems, that $\C$ is ($C\delta$)-semiconvex, can be weakened.
Indeed, in the course of the proof of Proposition~\ref{prop: C1b ptwise}, every application of Lemma~\ref{lem: u is close to ot soln} after the first will be to the $(C_3\delta)$-semiconvex sets $\{ \u_k \leq 1 \}$.
With respect to our main theorem, $\C = \overline{B}_1 - x_0$, which is as nice as imaginable.
\end{remark}

From this point forward, let $\N_r(E)$ denote the $r$-neighborhood of a set $E$.

\begin{lemma}
\label{lem: Gsub of u lives in an nhood of sub of ot soln}
Let $R \geq 3$, $\u : B_{1/R} \to (-R,R)$ be a $\G$-convex function such that $\pa_\G \u(B_{1/R}) \subset B_R$, and $\w \in C^1(B_{1/R})$ be convex.
Suppose that $\B_{R} \subset \mathfrak{g}$.
Fix $A \in \R^{n \times n}$ to be a symmetric matrix such that
\begin{equation}
\label{eqn: matrix is uni elliptic}
\frac{1}{K}\Id \leq A \leq K \Id
\end{equation}
for some $K \geq 1$. 
Define the ellipsoid
\[
\El(x_0,h) := \bigg{\{} x : \frac{1}{2}A(x-x_0) \cdot (x-x_0) \leq h \bigg{\}},
\]
and assume that $\El(x_0,h) \subset B_{1/R}$.
If
\[
\|\u - \w\|_{C^0(\El(x_0,h))} \leq \vep
\]
and 
\[
\|\G - x \cdot y + v\|_{C^2(\B_{R})} + \|\H - x \cdot y + u\|_{C^2(\B_{R})} \leq \del
\]
for small constants $\vep, \del > 0$, then
\[
\pa_\G \u(\El(x_0, h - \vep^{1/2})) \subset \N_{\del + K'(h\vep)^{1/2}} (\nab \w(\El(x_0, h))) \quad \forall 0 < \vep < h^2 \ll 1
\]
where $K' = K'(K) > 0$. 
\end{lemma}

\begin{proof}
Up to a change of coordinates, we can assume that $x_0 = 0$.
Let $\El(h) = \El(0,h)$ and
define
\[
\bar{\w}(x) := \w(x) + \vep + \vep^{1/2}(Ax \cdot x - 2h).
\]
By construction, $\bar{\w} \geq \u$ outside $\El(h)$ and $\bar{\w} \leq \u$ inside $\El(h -\vep^{1/2})$.
Therefore, if $\mathscr{G}_{x,y,v}$ is a $\G$-support for $\u$ at $x \in \El(h - \vep^{1/2})$, then $\mathscr{G}_{x,y,\bar{v}}$ will touch $\bar{\w}$ from below at a point $\bar{x} \in \El(h)$ for some $\bar{v} \geq v$.
Moreover, $(\bar{x},y,\bar{v}) \in \mathfrak{g}$ since $\bar{v} = \H(\bar{x},y,\bar{\w}(\bar{x})) \in (-R^2,R^2)$.
Hence,
\begin{equation}
\label{eqn: Gsub u inside Gsub wbar}
\pa_\G \u (\El(h - \vep^{1/2})) \subset \pa_\G \bar{\w}(\El(h)).
\end{equation}
Note that even though $\bar{\w}$ may not be $\G$-convex, we can still consider its $\G$-subdifferential;
it just might be empty at some points.
In particular, the equality $\pa_\G \bar{\w}(x) = \Gexp_{x,\bar{\w}(x)}(\nab \bar{\w}(x))$ still holds.
Thus, since $|\Gexp_{x,\bar{\w}(x)}(\nab \bar{\w}(x)) - \nab \bar{\w}(x)| \leq \del$ by assumption, we find that
\begin{equation}
\label{eqn: delta nhood of Gsub wbar}
\pa_\G \bar{\w}(\El(h)) \subset \N_{\del}(\nab\bar{\w}(\El(h)).
\end{equation}
From \eqref{eqn: matrix is uni elliptic}, we determine that $\El(h) \subset B_{(2Kh)^{1/2}}$.
And so since
\[
|\nab \bar{\w}(x)| \leq |\nab \w(x)| + 2\vep^{1/2}K|x|,
\]
recalling \eqref{eqn: matrix is uni elliptic}, it follows that
\begin{equation}
\label{eqn: subdiff w v subdiff wbar}
\N_{\del}(\nab \bar{\w}(\El(h))) \subset \N_{\del + 4K(Kh\vep)^{1/2}}(\nab \w(\El(h))).
\end{equation}
Finally, combining \eqref{eqn: Gsub u inside Gsub wbar}, \eqref{eqn: delta nhood of Gsub wbar}, and \eqref{eqn: subdiff w v subdiff wbar}, we deduce that
\[
\pa_\G \u(\El(h -\vep^{1/2})) \subset \N_{\del+K'(h\vep)^{1/2}}(\nab \w(\El(h)))
\]
with $K' = 4K^{3/2}$, as desired.
\end{proof}

\begin{proof}[Proof of Proposition~\ref{prop: C1b ptwise}]
The proof will be done in four steps.\\

{\it -- Step 1: $\u$ and its $\G$-sections are close to a strictly convex solution of a Monge--Amp\`{e}re equation and its sections.}\\\\
Using Lemma~\ref{lem: u is close to ot soln} and arguing exactly as in \cite{DF}, we find the existence of a strictly convex function $\w$ such that $\w(0) = \u(0) = 0$,
\begin{equation}
\label{eqn: u close from lem}
\|\u - \w\|_{C^0(B_{1/3})} \leq \om(\del),
\end{equation}
and
\begin{equation}
\label{eqn: MA equation}
\det (D^2 \w) = 1 \quad\text{in } B_{1/4}
\end{equation}
in the Alexandrov sense.
Furthermore, there exists a constant $K_0 = K_0(n) > 0$ such that
\begin{equation}
\label{eqn: pogo plus shauder}
\|\w\|_{C^3(B_{1/5})} \leq K_0 \qquad\text{and}\qquad \frac{1}{K_0} \Id \leq D^2 \w \leq K_0 \Id \quad\text{in }B_{1/5}.
\end{equation}
And so
\begin{equation}
\label{eqn: section height}
S(\w,h) := \{x : \w(x) \leq \nab \w(0) \cdot x + h \} \subset B_{(2K_0h)^{1/2}}.
\end{equation}
(Precisely, this is Step 1 in the proof of \cite[Theorem 4.3]{DF}, which uses \eqref{eqn: prop u close to parabola}.)
By \eqref{eqn: pogo plus shauder} and as $\u$ is semiconvex with a semiconvexity constant depending only on $\|D^2_x\G\|_{C^0(\B_3)}$, i.e., $\delta$ (\eqref{eqn: Gbar small}), we have that $\u - \w$ is semiconvex with a semiconvexity constant depending on dimension (recall that $\delta \ll 1$).
So using \eqref{eqn: u close from lem}, we deduce that
\begin{equation}
\label{eqn: grad minus Ggrad}
|\nab \w(0)| \leq K_1\om(\del)^{1/2}
\end{equation}
for some constant $K_1 = K_1(n) > 0$ (cf. \eqref{eqn: dist point and slope}, noticing that $-\nab \w(0) \in \pa^-(\u - \w + c|\cdot|^2)(0)$ for some $c > 0$ depending on $n$).

Define
\[
S_{\G}(\u,h) := \{ x : \u(x) \leq h \}.
\]
We claim that if $\del$ and $h$ are sufficiently small, then
\begin{equation}
\label{eqn: sections are comparable}
S(\w, h - K_2\om(\del)^{1/2}) \subset S_{\G}(\u,h) \subset S(\w, h + K_2\om(\del)^{1/2}) \Subset B_{1/6}
\end{equation}
where $K_2 = K_2(n) > 0$.
First, by \eqref{eqn: section height}, we can choose $\del$ and $h$ sufficiently small so that the last inclusion holds.
To conclude, let $x \in S(\w, h - K_2\om(\del)^{1/2})$.
Then, by \eqref{eqn: u close from lem}, recalling that $\u(0) = 0$, and from \eqref{eqn: grad minus Ggrad}, we deduce that
\[
\begin{split}
\u(x) &\leq \nab \w(0) \cdot x + h - K_2\om(\del)^{1/2} + \om(\del) \\
&\leq h + K_1\om(\del)^{1/2} - K_2\om(\del)^{1/2} + \om(\del)\\
&\leq h
\end{split}
\]
taking $K_2 = K_1 + 1$.
This proves the first inclusion; the proof of the second is analogous.
\\

{\it -- Step 2: The $\G$-sections of $\u$ and their images under $\pa_{\G} \u$ are close to ellipsoids with controlled eccentricity and $\u$ is close to a paraboloid at some small scale.}\\\\
We claim that for every small $\eta > 0$, there exist constants $h_0 = h_0(\eta,n) > 0$ and $\del = \del(h_0, \eta,n) > 0$ such that the following holds: there exists a symmetric matrix satisfying
\begin{equation}
\label{eqn: thm matrix is uni elliptic}
\frac{1}{K_3} \Id \leq A \leq K_3 \Id,
\end{equation}
\begin{equation}
\label{eqn: thm matrix has det 1}
\det(A) = 1,
\end{equation}
\begin{equation}
\label{eqn: sections are round}
AB_{h_0^{1/2}/3} \subset S_{\G}(\u,h_0) \subset AB_{3h_0^{1/2}},
\end{equation}
and
\begin{equation}
\label{eqn: Gsubs are round}
A^{-1}B_{h_0^{1/2}/3} \subset \pa_{\G} \u(S_{\G}(\u,h_0)) \subset A^{-1}B_{3h_0^{1/2}}.
\end{equation}
Moreover,
\begin{equation}
\label{eqn: thm u is close to smooth fn}
\bigg{\|} \u - \frac{1}{2}|A^{-1}x|^2\bigg{\|}_{C^0\big(AB_{3h_0^{1/2}}\big)} \leq \eta h_0.
\end{equation}
Here, $K_3 = K_3(n) > 0$.

Let
\[
A := [D^2\w(0)]^{-1/2}.
\]
With $A$ defined in this way, using \eqref{eqn: pogo plus shauder} and \eqref{eqn: MA equation}, we see that \eqref{eqn: thm matrix is uni elliptic} and \eqref{eqn: thm matrix has det 1} hold.

Notice that \eqref{eqn: sections are round} is equivalent to
\[
\El(h_0/18) \subset S_{\G}(\u,h_0) \subset \El(9h_0/2)
\]
where
\[
\El(h) := \bigg{\{} x : \frac{1}{2}D^2\w(0)x\cdot x \leq h \bigg{\}}.
\] 
Now from \eqref{eqn: pogo plus shauder}, we deduce
\begin{equation}
\label{eqn: ellipse/section/ball}
\El(h) \subset B_{(2K_0h)^{1/2}}.
\end{equation}
Consequently, 
\[
\El(h) \subset S(\w,h+ K_0(2K_0h)^{3/2})
\qquad\text{and}\qquad
S(\w,h) \subset \El(h+ K_0(2K_0h)^{3/2}).
\]
The second inclusion follows from \eqref{eqn: section height}.
Thus, applying \eqref{eqn: sections are comparable}, we see that 
\[
\El(h_0/18) \subset S_{\G}(\u,h_0)
\]
provided that $h_0$ and $\del$ are sufficiently small depending only on $n$.
On the other hand, applying \eqref{eqn: sections are comparable}, we see that
\[
S_{\G}(\u,h_0) \subset \El(9h_0/2)
\]
so long as $\del$ and $h_0$ are sufficiently small, again, depending only on $n$.
Whence, \eqref{eqn: sections are round} holds, as desired.
More generally, for every $c < 1$ and $C > 1$, we can find $\del$ and $h_0$ sufficiently small so that
\begin{equation}
\label{eqn: Gsection comp to ellipse}
\El(ch_0) \subset S_{\G}(\u,h_0) \subset \El(Ch_0).
\end{equation}

Let us now prove \eqref{eqn: Gsubs are round}.
To do this, we consider the $\G$-transform of $\u$ and the Legendre transform of $\w$.
Specifically,
\[
\v(y) := \sup_{x \in B_{1/5}} \H(x,y,\u(x)) \qquad\text{and}\qquad \w^*(y) := \sup_{x \in B_{1/5}} \{ x \cdot y - \w(x)\}.
\]
Notice that by \eqref{eqn: Gbar small} and \eqref{eqn: u close from lem},
\[
\|\v - \w^*\|_{C^0(B_{1/3})} \leq \om(\del) + \del \leq 2\om(\del).
\]
Also, observe that
\begin{equation}
\label{eqn: legendre ids}
\nab \w^* = [\nab \w]^{-1} \qquad\text{and}\qquad D^2\w^*(\nab \w(x)) = [D^2\w(x)]^{-1}.
\end{equation}
Moreover, from \eqref{eqn: pogo plus shauder}, $\w^*$ is uniformly convex and of class $C^3$ in the open set $\nab \w(B_{1/5})$.
Let
\[
\El^*(h) := \bigg{\{} y : \frac{1}{2}D^2 \w^*(\nab \w(0))(y - \nab \w(0))\cdot (y-\nab \w(0)) \leq h \bigg{\}}.
\]
Thanks to \eqref{eqn: pogo plus shauder}, for every $c < 1$ and $C > 1$, we can find $h_0$ sufficiently small so that
\begin{equation}
\label{eqn: grad map w ellipsiod v w sections}
\nab \w(\El(ch_0)) \subset \El^*(h_0) \subset \nab \w (\El(Ch_0)).
\end{equation}
Indeed, using \eqref{eqn: legendre ids} and \eqref{eqn: pogo plus shauder}, we find that
\[
|D^2 \w^*(\nab \w(0))(\nab \w(x)- \nab \w(0))\cdot(\nab \w(x) - \nab \w(0)) - D^2 \w(0)x\cdot x| \leq 2K_0|x|^3 + K_0^3|x|^4
\]
and that $|x| \leq K_0(2K_0h)^{1/2}$ whenever $\nab \w(x) \in \El^*(h)$.
Combining these last two inequalities and \eqref{eqn: ellipse/section/ball} proves \eqref{eqn: grad map w ellipsiod v w sections}.
Then, by \eqref{eqn: Gsection comp to ellipse} with $C = 3/2$, Lemma~\ref{lem: Gsub of u lives in an nhood of sub of ot soln} applied with $h = 3h_0/2 + \om(\del)^{1/2}$, and \eqref{eqn: grad map w ellipsiod v w sections}, we determine that 
\[
\pa_{\G} \u (S_{\G}(\u,h_0)) \subset \N_{K_0''\om(\del)^{1/2}}(\nab \w(\El(2h_0))) \subset \El^*(7h_0/2)
\]
for $\del$ and $h_0$ sufficiently small, depending only on dimension.
Hence, recalling \eqref{eqn: grad minus Ggrad} and choosing $\del$ sufficiently small, we find that the second inclusion in \eqref{eqn: Gsubs are round} holds.
In order to conclude, we must show that
\[
\El^*(h_0/16) \subset \pa_{\G} \u (S_{\G}(\u,h_0)).
\]
The first inclusion in \eqref{eqn: Gsubs are round} then follows, again, by \eqref{eqn: grad minus Ggrad} and choosing $\del$ sufficiently small. 
Since
\[
E \subset \pa_{\G}\u(\pa_{\H} \v(E)) \quad \forall E,
\]
considering \eqref{eqn: Gsection comp to ellipse}, we see it suffices to show that
\[
\pa_{\H} \v(\El^*(h_0/16)) \subset \El(h_0/3).
\]
Applying Lemma~\ref{lem: Gsub of u lives in an nhood of sub of ot soln} to $\v$ and $\w^*$, provided that $\del$ is small enough, we determine that
\[
\pa_\H \v(\El^*(h_0/16)) \subset \N_{K_0'''\om(\del)^{1/2}}(\nab \w^*(\El^*(h_0/8))) \subset \El(h_0/3).
\]
Here, we have used \eqref{eqn: grad map w ellipsiod v w sections} and \eqref{eqn: legendre ids} for the second inclusion.
Thus, \eqref{eqn: Gsubs are round} indeed holds after taking $\del$ and $h_0$ sufficiently small.

Finally, from \eqref{eqn: u close from lem}, \eqref{eqn: pogo plus shauder}, \eqref{eqn: grad minus Ggrad}, and \eqref{eqn: thm matrix is uni elliptic}, we see that
\[
\begin{split}
\bigg{\|} \u - \frac{1}{2}|A^{-1}x|^2\bigg{\|}_{C^0\big(AB_{3h_0^{1/2}}\big)}  &\leq \|\u - \w\|_{C^0\big(AB_{3h_0^{1/2}}\big)} + \bigg{\|} \w - \frac{1}{2}|A^{-1}x|^2\bigg{\|}_{C^0\big(AB_{3h_0^{1/2}}\big)}\\
&\leq \om(\del) + \|\nab \w(0)\cdot x\|_{C^0\big(AB_{3h_0^{1/2}}\big)} + \|K_0|x|^3\|_{C^0\big(AB_{3h_0^{1/2}}\big)}\\
&\leq \om(\del) + 3K_1K_3\om(\del)^{1/2}h_0^{1/2} + 27K_0K_3^3h_0^{3/2} \\
&\leq \eta h_0
\end{split}
\]
where the last inequality follows after first choosing $h_0$ sufficiently small and then choosing $\del$ even smaller.
\\

{\it -- Step 3: An iterative construction.}\\\\
Set
\[
\tilde{x} := \frac{1}{h_0^{1/2}}A^{-1}x, \qquad \tilde{y} := \frac{1}{h_0^{1/2}}Ay, \qquad \tilde{v} := \frac{1}{h_0}v, \qquad\text{and}\qquad \tilde{u} := \frac{1}{h_0}u.
\]
Define
\[
\G_1(\tilde{x},\tilde{y},\tilde{v}) := \frac{\G(h_0^{1/2}A\tilde{x},h_0^{1/2}A^{-1}\tilde{y},h_0 \tilde{v})}{h_0} \qquad\text{and}\qquad
\u_1(\tilde{x}) := \frac{\u(h_0^{1/2}A\tilde{x})}{h_0}.
\]
Similarly, let
\[
\H_1(\tilde{x},\tilde{y},\tilde{u}) := \frac{\H(h_0^{1/2}A\tilde{x},h_0^{1/2}A^{-1}\tilde{y},h_0 \tilde{u})}{h_0}.
\] 
Recalling \eqref{eqn: extra 1}, \eqref{eqn: extra 2}, \eqref{eqn: Gbar small}, and that we are considering parabolically quadratic rescalings, we see that
\[
\|\G_1 - \tilde{x} \cdot \tilde{y} + \tilde{v}\|_{C^{2,\alpha}(\B_3)} + \|\H_1 - \tilde{x} \cdot \tilde{y} + \tilde{u}\|_{C{2,\alpha}(\B_3)} \leq \del
\]
provided $h_0 = h_0(K_3,\alpha) \ll 1$.
In addition, thanks to \eqref{eqn: sections are round} and \eqref{eqn: Gsubs are round}, we have the inclusions
\[
B_{1/3} \subset \C_1, \K_1 \subset B_3
\]
where, recalling that $\G_1(\cdot,0,0) \equiv 0$,
\[
\C_1 := S_{\G_1}(\u_1,1) = \{\u_1 \leq 1 \} \qquad\text{and}\qquad \K_1 := \pa_{\G_1} \u_1(S_{\G_1}(\u_1,1)).
\]
Arguing as in Proposition~\ref{prop: comparison}, we find that $\C_1$ is a closed, $(C_3 \delta)$-semiconvex set.
Furthermore, rewriting \eqref{eqn: thm u is close to smooth fn} yields
\[
\bigg\|\u_1 - \frac{1}{2}|\tilde{x}|^2 \bigg\|_{C^0(B_3)} \leq \eta.
\]
Now let
\[
f_1(\tilde{x}) := f(h_0^{1/2}A\tilde{x})\1_{\C_1} \qquad\text{and}\qquad g_1(\tilde{y}) := g(h_0^{1/2}A^{-1}\tilde{y})\1_{\K_1}.
\]
Recalling that $\det(A) = 1$ and arguing as in the proof of Theorem~\ref{thm: main}, we determine that
\[
(\T_{\u_1})_\#f_1 = g_1.
\]
Finally, using \eqref{eqn: prop densities close to constant}, it follows that
\[
\|f_1 - 1\|_{L^\infty(\C_1)} + \|g_1 - 1\|_{L^\infty(\K_1)} \leq \del.
\]
Hence, we can apply Steps 1 and 2 to $\u_1$ and find a symmetric matrix $A_1$ such that
\[
\frac{1}{K_3} \Id \leq A_1 \leq K_3 \Id,
\]
\[
\det(A_1) = 1,
\]
\[
A_1B_{h_0^{1/2}/3} \subset S_{\G_1}(\u_1,h_0) \subset A_1B_{3h_0^{1/2}},
\]
\[
A_1^{-1}B_{h_0^{1/2}/3} \subset \pa_{\G_1} \u(S_{\G_1}(\u_1,h_0)) \subset A_1^{-1}B_{3h_0^{1/2}},
\]
and
\[
\bigg{\|} \u_1 - \frac{1}{2}|A_1^{-1}\tilde{x}|^2\bigg{\|}_{C^0\big(A_1B_{3h_0^{1/2}}\big)} \leq \eta h_0.
\]
(The set $S_{\G_1}(\u_1,h) = \{ \u_1 \leq h \}$.)

We can continue, iteratively constructing
\[
\G_{k+1}(\tilde{x},\tilde{y},\tilde{v}) := \frac{\G_k(h_0^{1/2}A_k \tilde{x},h_0^{1/2}A_k^{-1} \tilde{y}, h_0 \tilde{v})}{h_0} \qquad\text{and}\qquad
\u_{k+1}(\tilde{x}) := \frac{\u_k(h_0^{1/2}A_k \tilde{x})}{h_0}
\]
where $A_k$ is the symmetric matrix constructed in the the $k$-th iteration.
In turn, setting
\[
M_k := A_1 \cdot \hdots \cdot A_k,
\]
we have a sequence of symmetric matrices such that
\begin{equation}
\label{eqn: iteration matrices}
\frac{1}{K_3^k} \Id \leq M_k \leq K_3^k \Id,
\end{equation}
\[
\det(M_k) = 1,
\]
and
\begin{equation}
\label{eqn: iteration sections}
M_kB_{h_0^{k/2}/3} \subset S_{\G_1}(\u_1,h_0^k) \subset M_kB_{3h_0^{k/2}}.
\end{equation}\

{\it -- Step 4: $C^{1,\beta}(0)$-regularity.}\\\\
Let $\beta \in (0,1)$.
By \eqref{eqn: iteration matrices} and \eqref{eqn: iteration sections}, we find that
\begin{equation}
\label{eqn: C1b}
B_{(h_0^{1/2}/3K_3)^k} \subset S_{\G_1}(\u_1,h_0^k) \subset B_{(3K_3h_0^{1/2})^{k}}.
\end{equation}
Defining $r_0 := h_0^{1/2}/3K_3$ and recalling that $\G_1(\cdot,0,0) \equiv 0$, it follows that
\[
\|\u_1\|_{C^0(B_{r_0^k})} \leq h_0^k = (3K_3 r_0)^{2k} \leq r_0^{(1+\beta)k}
\]
provided $h_0$ (and so $r_0$) is sufficiently small.
In other words, $\u_1$ and $\u$ are $C^{1,\beta}$ at the origin.
\end{proof}

With Proposition~\ref{prop: C1b ptwise} in hand, let us now prove Theorem~\ref{thm: close implies C1b}.
The proof amounts to a change of variables.

\begin{proof}[Proof of Theorem~\ref{thm: close implies C1b}]
Let $x_0 \in B_{1/6}$ and $y_0 \in \pa_\G \u(x_0)$.
By \eqref{eqn: thm G close to linear}, observe that
\[
|x_0 - y_0| \leq |x_0 - p_0| + |p_0 - y_0| \leq |x_0 - p_0| + \del_0
\]
where
\[
p_0 := D_x\G(x_0,y_0,\H(x_0,y_0,\u(x_0))).
\]
As $\u$ is semiconvex with a semiconvexity constant depending only on $\|D^2_x\G\|_{C^0(\B_4)} \leq \delta$ (recall \eqref{eqn: thm G close to linear}), there exists a $c > 0$ such that
\[
\w(x) := \u(x) - \frac{1}{2}|x|^2 + c|x-x_0|^2
\]
is convex.
By construction, $p_0 - x_0 \in \pa^-\w(x_0)$.
And so using \eqref{eqn: thm u close to parabola} and since $x_0 + \eta^{1/2}_0\e \in \C$ provided that $\eta_0 < 1/9$, for example, we find that
\[
(p_0 - x_0) \cdot \e \leq \frac{\w(x_0+\eta_0^{1/2}\e) - \w(x_0)}{\eta_0^{1/2}} \leq (2+c)\eta_0^{1/2} \quad \forall \e \in \mathbb{S}^{n-1}.
\]
In turn,
\begin{equation}
\label{eqn: dist point and slope}
|x_0 - y_0| \leq C(\del_0 + \eta_0^{1/2}).
\end{equation}
Set 
\[
u_0 := \u(x_0) \qquad\text{and}\qquad v_0 := \H(x_0,y_0,\u(x_0)).
\]
From \eqref{eqn: thm u close to parabola} and using \eqref{eqn: thm G close to linear} and \eqref{eqn: thm bounds on Ci}, we deduce that
\begin{equation}
\label{eqn: points for first cov}
|u_0| \leq \frac{1}{72} + \eta_0 \qquad\text{and}\qquad |v_0| \leq \frac{1}{3} + \frac{1}{72} + \eta_0 + \del_0.
\end{equation}
Define
\[
M_0 := \E(x_0,y_0,v_0) \qquad\text{and}\qquad a_0 := -D_v\G(x_0,y_0,v_0),
\]
where $\E$ is as defined in \eqref{def: matrix E}.
Observe that \eqref{eqn: thm G close to linear} implies that
\begin{equation}
\label{eqn: size of hat M and hat a}
|a_0 - 1|, |a_0^{-1} - 1| \leq \del_0 \qquad\text{and}\qquad |M_0 - \Id|, |M_0^{-1} - \Id| \leq 2\del_0.
\end{equation}

Now consider the change of variables
\[
\bar{x} := x-x_0,\qquad \bar{y} := M_0(y-y_0), \qquad \bar{v} := a_0(v - v_0), \qquad\text{and}\qquad \bar{u}:= u - u_0.
\] 
Define
\[
\bar{\C} := \C - x_0 \quad\text{and}\quad  \bar{\K} := M_0(\K - y_0);
\]
and set
\[
\bar{f}(\bar{x}) := f(\bar{x}+x_0) \quad\text{and}\quad \bar{g}(\bar{y}) := \det(M_0^{-1})g(M_0^{-1}\bar{y}+y_0).
\]
Then, from \eqref{eqn: thm bounds on Ci}, \eqref{eqn: dist point and slope}, and \eqref{eqn: size of hat M and hat a}, we see that
\[
B_{1/3} \subset \bar{\C}, \bar{\K} \subset B_3
\]
if $\del_0$ and $\eta_0$ are sufficiently small.
From \eqref{eqn: size of hat M and hat a}, we have that $|\det (M_0) - 1| \leq (1+4n)\del_0$ if $\del_0$ is sufficiently small.
Thus,
\begin{equation*}
\|\bar{f} - 1\|_{L^\infty(\bar{\C})} + \|\bar{g} - 1\|_{L^\infty(\bar{\K})} \leq 4(1+n)\del_0,
\end{equation*}
recalling \eqref{eqn: thm densities close to constant}.
Let
\[
\bar{\G}(\bar{x},\bar{y},\bar{v}) := \G(x,y,v-v_0+\H(x_0,y,u_0)) - \G(x,y_0,v_0).
\]
Notice that the dual of $\bar{\G}$ is
\[
\bar{\H}(\bar{x},\bar{y},\bar{u}) = a_0(\H(x,y,u - u_0+\G(x,y_0,v_0)) - \H(x_0,y,u_0)),
\]
and $\bar{\H}$ is well-defined on $\B_3$ by the assumption that $\B_4 \subset \mathfrak{h}$ and our estimates on $x_0$, $y_0$, $u_0$, $v_0$, $a_0$, and $M_0$.
Similarly, from \eqref{eqn: thm G close to linear}, \eqref{eqn: dist point and slope}, \eqref{eqn: points for first cov}, and \eqref{eqn: size of hat M and hat a}, it follows that
\begin{equation*}
\|\bar{\G} - \bar{x} \cdot \bar{y} + \bar{v}\|_{C^{2,\alpha}(\B_3)} + \|\bar{\H} - \bar{x} \cdot \bar{y} + \bar{u}\|_{C^{2,\alpha}(\B_3)}\leq C_0\del_0.
\end{equation*}
In particular,
\begin{equation*}
\bar{\G}(\cdot,0,0) = \bar{\G}(0,\cdot,0) = \bar{\H}(\cdot,0,0) = \bar{\H}(0,\cdot,0) \equiv 0.
\end{equation*}
Also, computations show that
\begin{equation*}
D_{\bar{v}}\bar{\G}(0,0,0) = D_{\bar{u}}\bar{\H}(0,0,0) = -1
\qquad\text{and}\qquad
D_{\bar{x}\bar{y}} \bar{\G}(0,0,0) = D_{\bar{x}\bar{y}} \bar{\H}(0,0,0) = \Id.
\end{equation*}
Set
\[
\bar{\del} := \min\{ 4(1+n)\del_0, C_0\del_0\}.
\]
Finally, define
\[
\bar{\u}(\bar{x}) : = \u(x) - \G(x,y_0,v_0).
\]
Arguing as in the proof of Theorem~\ref{thm: main}, we see that $\bar{\u}$ is $\bar{\G}$-convex and $(\T_{\bar{u}})_\# \bar{f} = \bar{g}$.
Furthermore, from \eqref{eqn: thm u close to parabola}, \eqref{eqn: thm G close to linear}, and \eqref{eqn: dist point and slope}, we determine that
\[
\bigg{\|}\bar{\u} - \frac{1}{2}|\bar{x}|^2\bigg{\|}_{C^0(B_3)} \leq 2\eta_0 + 2\del_0 + C(\del_0 + \eta_0^{1/2}) =: \bar{\eta}.
\]
Indeed, recalling that $v_0 := \H(x_0,y_0,\u(x_0))$,
\[
\begin{split}
\bigg{|}\bar{\u} - \frac{1}{2}|\bar{x}|^2\bigg{|}
&= \bigg{|}\u(x) - \G(x,y_0,v_0) - \frac{1}{2}x\cdot x + x \cdot x_0 + \u(x_0) -\u(x_0) +\frac{1}{2} x_0 \cdot x_0 - x_0 \cdot x_0 \bigg{|}\\
&\leq 2\eta_0 + |- \G(x,y_0,v_0) + x \cdot x_0 + \u(x_0) - x_0 \cdot x_0 |\\
&= 2\eta_0 + |x\cdot y_0  - v_0 - \G(x,y_0,v_0) - x \cdot y_0 + v_0 - x_0 \cdot y_0 + \u(x_0) \\
&\hspace{0.91in}+ x_0 \cdot y_0 + x \cdot x_0  - x_0 \cdot x_0 |\\
&\leq 2\eta_0 + 2\del_0 + (|x| + |x_0|)|x_0 -y_0|\\
&\leq 2\eta_0 + 2\del_0 + C(\del_0 + \eta_0^{1/2}).
\end{split}
\]

In summary, we see that $\bar{\u}$, $\bar{\G}$, $\bar{\H}$, $\bar{f}$, $\bar{g}$, $\bar{\C}$, and $\bar{\K}$ satisfy the hypotheses of Proposition~\ref{prop: C1b ptwise}.
Hence, taking $\del_0$ and $\eta_0$, in turn, $\bar{\del}$ and $\bar{\eta}$, sufficiently small, we find that $\u \in C^{1,\beta}(B_{1/6})$, as desired.
\end{proof}

An important corollary of Theorem~\ref{thm: close implies C1b} is a strict $\G$-convexity estimate for $\u$ in $B_{1/6}$.

\begin{corollary}
\label{cor: strict Cconvexity}
Under the hypotheses of Theorem~\ref{thm: close implies C1b}, we find that $\u$ is strictly $\G$-convex in $B_{1/6}$.
More precisely, for all $\sigma > 2$, there exist constants $\eta_0, \del_0 > 0$, depending on $\sigma$ and dimension, such that for all $x_0 \in B_{1/6}$, we have that
\begin{equation}
\label{eqn: strict Gconvex}
\inf_{\pa B_d(x_0)} \{ \u - \mathscr{G}_{x_0,y_0,v_0} \} \geq c_0d^\sigma \quad \forall d \leq \dist(x_0,\pa B_{1/6})
\end{equation}
for some constant $c_0 = c_0(\sigma,n) > 0$.
Here, $y_0 \in \pa_\G \u(x_0)$ and $v_0 := \H(x_0,y_0,\u(x_0))$.
\end{corollary}

\begin{proof}
Let $\u_1$ be as in Proposition~\ref{prop: C1b ptwise} and set $d_0 := 3K_3h_0^{1/2}$.
Then, from \eqref{eqn: C1b}, we deduce that
\[
\inf_{\pa B_d} \u_1 \geq d_0^\sigma d^\sigma \quad \forall d \leq d_0
\]
provided $d_0$ is sufficiently small depending on $\sigma$ and dimension.
In turn,
\[
\inf_{\pa B_{d}(x_0)} \{ \u - \mathscr{G}_{x_0,y_0,v_0} \} \geq h_0d_0^\sigma \bigg(\frac{h_0^{1/2}d}{K_3}\bigg)^\sigma \geq c_0d^\sigma \quad \forall d \leq \dist(x_0,\pa B_{1/6})
\]
taking $c_0 = c_0(\sigma,n) > 0$ sufficiently small, as desired. 
\end{proof}

From the strict $\G$-convexity of $\u$ in $B_{1/6}$, we deduce that $\T_\u(B_{1/6})$ is open, a key fact used in the proof of Theorem~\ref{thm: main}.

\begin{corollary}
\label{cor: Tu is locally open}
Under the hypotheses of Theorem~\ref{thm: close implies C1b}, we have that $\T_\u(B_{1/6})$ is open.
\end{corollary}

\begin{proof}
Since $\u$ is differentiable in $B_{1/6}$, we have that $\T_\u(B_{1/6}) = \pa_\G \u(B_{1/6})$.
We show that for each $x_0 \in B_{1/6}$, there exists an $\vep_0 > 0$ such that for any $y \in B_{\vep_0}(y_0)$, the map 
\[
z \mapsto \H(z,y,\u(z))
\]
has a local maximum at some point $x \in B_{1/6}$.
Here, $\{y_0\} := \pa_\G \u(x_0)$.
If so, then
\[
\nab \u(x) = D_x \G(x,y,\H(x,y,\u(x)));
\]
that is, $\{y\} = \pa_\G \u(x)$ and $B_{\vep_0}(y_0) \subset \T_\u(B_{1/6})$, as desired.
To this end, let $d > 0$ be such that $B_d(x_0) \subset B_{1/6}$ and
\[
x \in \argmax_{z \in \overline{B}_d(x_0)} \H(z,y,\u(z)).
\]
Since $\u(x) = \G(x,y,\H(x,y,\u(x)))$ and $\G$ is decreasing in $v$, we observe that
\[
\begin{split}
\u(x) - \mathscr{G}_{x_0,y_0,v_0}(x) &\leq \G(x,y,\H(x_0,y,\u(x_0))) - \G(x,y_0,\H(x_0,y_0,\u(x_0))) \\
&= \G(x,y,\H(x_0,y,\u(x_0))) - \G(x,y_0,\H(x_0,y,\u(x_0))) \\ 
&\hspace{1.0cm} + \G(x,y_0,\H(x_0,y,\u(x_0)))- \G(x,y_0,\H(x_0,y_0,\u(x_0))) \\
&\leq C\vep_0
\end{split}
\]
with $C := \|\G\|_{C^1(\B_4)}(1 + 2\|\H\|_{C^1(\B_4)})$.
Hence, taking $\vep_0 < c_0d^\sigma/C$, we see that
\[
\u(x) - \mathscr{G}_{x_0,y_0,v_0}(x) < c_0 d^\sigma, 
\]
which, recalling \eqref{eqn: strict Gconvex}, implies that $x$ lives inside $B_d(x_0)$ and not on its boundary.
\end{proof}


\section{Higher Regularity}
\label{sec: HR}

Here, we prove a higher regularity version of Theorem~\ref{thm: close implies C1b}.
To do this, we will need a more refined comparison-type principle than the one established in Lemma~\ref{lem: u is close to ot soln}.
The comparison-type principle in this section makes use of a change of variables formula for the $\G$-exponential map, Lemma~\ref{lem: cov Gexp}, and the coincidence of the $\G$-subdifferential of $\u$ at $x$ and the $\G$-exponential map at $(x,\u(x),\nab \u(x))$ when $\u$ is differentiable, Remark~\ref{rmk: local to global}.

Given a $\G$-convex function $\w$ on an open set $\O$, we have defined $\T_\w(x) := \Gexp_{x,\w(x)}(\nab \w(x))$.
Yet even when $\w$ is not $\G$-convex, we may still consider $\T_\w(x)$ if $\{ (x,\w(x),\nab \w(x)) : x \in \O \} \subset \dom \Gexp$.

\begin{lemma}
\label{lem: cov Gexp}
Let $\O \subset \R^n$ be open, $\w \in C^2(\O)$, and $\{ (x,\w(x),\nab \w(x)) : x \in \O \} \subset \dom \Gexp$.
If
\[
D^2 \w(x) - D^2_x \G(x,\T_\w(x),\H(x,\T_\w(x),\w(x))) \geq 0 \quad \forall x \in \O,
\]
then for every Borel set $E \subset \O$,
\[
|\T_\w(E)| \leq \int_E \frac{\det(D^2 \w(x) - D^2_x \G(x,\T_\w(x),\H(x,\T_\w(x),\w(x))))}{|\det (\E(x,\T_\w(x),\H(x,\T_\w(x),\w(x))))|} \, \d x.
\]
In addition, if the map $\T_\w$ is injective, then equality holds.
\end{lemma}

\begin{proof}
After differentiating the identity
\[
\nab \w(x) = D_x \G(x,\T_\w(x),\H(x,\T_\w(x),\w(x))),
\]
we see that the Jacobian determinant of the $C^1$ map $x \mapsto \T_\w(x)$ is the integrand above.
Thus, applying the Area Formula (see, e.g., \cite{EG}) concludes the proof.
\end{proof}

\begin{remark}
\label{rmk: local to global}
Recall, by \eqref{eqn: Gsub in Gexp}, that if $\u$ is differentiable at $x$ and $\G$-convex, then
\[
\pa_\G \u(x) = \{ \Gexp_{x,\u(x)}(\nab \u(x)) \} = \{ \T_\u(x) \}.
\]
\end{remark}

Let $\co[E]$ denote the convex hull of the set $E$.
Also, recall that $\N_r(E)$ denotes the $r$-neighborhood of a set $E$.
The following comparison-type principle compares $\G$-convex functions of class $C^1$ and smooth solutions of Monge--Amp\`{e}re equations.

\begin{proposition}
\label{prop: comparison}
Let $R \geq 3$ and $\u \geq 0$ be a $\G$-convex function of class $C^1$ such that $\u(0) = 0$ and
\begin{equation}
\label{eqn: comparison subdiff section inclusion}
B_{1/R} \subset S:= \{ \u \leq 1 \}  \subset B_R.
\end{equation}
Assume $\B(R,S) := B_{2R} \times \T_\u(S) \times (-R^3,R^3) \subset \mathfrak{g}, \mathfrak{h}$ and $B_{2R} \times (-R^3,R^3) \times \nab \u(S) \Subset \dom \Gexp$.
Suppose that $f$ and $g$ are two densities such that $(\T_\u)_\# f = g$ and
\begin{equation}
\label{eqn: prop densities close to cst}
\bigg{\|} \frac{f}{\l_1} - 1 \bigg{\|}_{C^0(S)} + \bigg{\|} \frac{g}{\l_2} - 1 \bigg{\|}_{C^0(\T_\u(S))} \leq \vep
\end{equation}
for some constants $\l_1/\l_2 \in (1/2,2)$ and $\vep \in (0,1/4)$.
Furthermore, assume that
\begin{equation}
\label{eqn: prop G close to linear}
\|\G - x \cdot y + v\|_{C^2(\B(R,S))} + 
\|\H - x \cdot y + u\|_{C^2(\B(R,S))}
\leq \del.
\end{equation}
Then, there exists constants $\gamma = \gamma(n,R) \in (0,1)$ and $\del_1 = \del_1(n,R) > 0$ such that the following holds: if $\w$ is convex and satisfies
\[
\begin{cases}
\det (D^2 \w) = \l_1/\l_2 &\text{in } \N_{\del^\gamma}(\co[S])\\
\w = 1 &\text{on } \pa \N_{\del^\gamma}(\co[S]),
\end{cases}
\]
then
\begin{equation}
\label{eqn: prop u close to MA soln}
\|\u - \w\|_{C^0(S)} \leq K(\vep + \del^{\gamma/n})
\end{equation}
provided $\del \leq \del_1$.
Here, $K = K(n,R) > 0$.
\end{proposition}

The proof of Proposition~\ref{prop: comparison} follows the proof of \cite[Proposition~5.2]{DF}.
Yet because the map $\T_\u$ depends on $x$, $\u$, and $\nab \u$ and not just on $x$ and $\nab \u$, the argument is more delicate.

\begin{proof}
Recall that $\u + \delta|x|^2$ is convex by \eqref{eqn: prop G close to linear}.
Thus, as $\u(0) = 0$, $\u = 1$ on $\pa S$, and $S \subset B_R$, using \eqref{eqn: first derivative}, it follows that
\begin{equation}
\label{eqn: DxG bounded below}
|D_x \G(x,y,v)| = |\nab \u(x)| \geq |\nab \u(x) + 2\delta x| - 2\delta|x| \geq \frac{1}{R} - 2\delta R \geq \frac{1}{2R} \quad \forall x \in \pa S
\end{equation}
provided that $\del$ is small enough.
Here, $y := \T_\u(x)$ and $v := \H(x,y,\u(x))$.
Now consider
\[
\mathcal{S} := \bigcap_{x \in \pa S} E_{x}
\]
where
\[
E_{x} := \{ z \in B_R : \G(z,y,v) \leq 1 \}.
\]
Clearly, $S \subset \mathcal{S}$.
Let $z \notin S$ and $x \in \pa S$ be a point such that $\dist(z,\pa S) = |x-z| > 0$.
If $|z-x| < 1/\delta R$, then using \eqref{eqn: prop G close to linear} and \eqref{eqn: DxG bounded below}, we find that
\[
\G(z,y,v) - 1 \geq|\nab \u(x)||z-x| - \frac{\delta}{2}|z-x|^2 > 0,
\]
and $z \notin \mathcal{S}$.
On the other hand, if $|z-x| \geq 1/\delta R$, then by \eqref{eqn: prop G close to linear} and \eqref{eqn: DxG bounded below}, we have that
\[
\begin{split}
\G(z,y,v) - 1 &\geq |z-x||\nab \u(x)| + (z-x)\cdot(y-D_x\G(x,y,v)) -2\del\\
&\geq \frac{1}{2\delta R^2} - 2R\del - 2\del.
\end{split}
\]
And so $\G(z,y,v) - 1 > 0$ and $z \notin \mathcal{S}$ provided that $\del$ is sufficiently small.
In turn,
\[
\mathcal{S} = S
\]
if $\del > 0$ is sufficiently small depending only on $R$.
It follows that $S$ is a ($C_R\delta$)-semiconvex set.
Now arguing exactly as in \cite[Proposition 5.2]{DF}, we have that 
\begin{equation}
\label{eqn: oscill}
\osc_S \w \leq K_{R,n},
\end{equation}
\begin{equation}
\label{eqn: bounds on w}
1 - K_{R,n} \del^{\gamma/n} \leq \w < 1 \quad\text{on } \pa S,
\end{equation}
and
\begin{equation}
\label{eqn: bounds on D2w}
D^2 \w \geq \frac{\del^{\gamma/\tau}}{K_{R,n}} \Id \quad\text{in } \co[S]
\end{equation}
for some constants $K_{R,n}$ and $\tau > 0$ depending only on dimension and $R$.

Define
\[
\w^+ := (1 + 3\vep + 2\del^{1/2})\w - 3\vep - 2\del^{1/2}
\]
and
\[
\w^- := \bigg(1 - 3\vep - \frac{\del^{1/2}}{2}\bigg)\w + 3\vep + \frac{\del^{1/2}}{2} + K_{R,n} \del^{\gamma/n}.
\]
We claim that if $\gamma$ is sufficiently small, then $\w^- \geq \u \geq \w^+$ in $S$.
If so, then \eqref{eqn: oscill} will imply \eqref{eqn: prop u close to MA soln}, as desired.

Choose $\gamma := \tau/4$.
By \eqref{eqn: bounds on w}, we have that $\w^- > \u > \w^+$ on $\pa S$.
We first show that $\u \geq \w^+$ in $S$.
Suppose not.
Then, as $\u > \w^+$ on $\pa S$, we see that
\[
\emptyset \neq Z := \{ \u < \w^+ \} \Subset S.
\]
Thanks to \eqref{eqn: bounds on D2w} and \eqref{eqn: prop G close to linear}, we have that
\begin{equation}
\label{eqn: uni convex w plus}
D^2 \w^+ > D^2\w \geq \frac{\del^{1/4}}{K_{R,n}}\Id > \del \Id \geq \|D^2_x\G\|_{C^0(\B(R,S))} \Id \quad \text{in } \co[S]
\end{equation}
provided that $\del$ is sufficiently small depending on $R$ and $n$.
Notice that $\w^+(Z) \subset (-2R^2,1)$.\footnote{\,By \eqref{eqn: bounds on w} and the convexity of $\w^+$, we see that $\w^+ < 1$ in $Z \subset S$.
The inclusion $\w^+(Z) \subset (-2R^2,1)$ then follows from considering the lower barrier (for $\w$)
\[
\frac{\l_1^{1/n}}{2\l_2^{1/n}}(|x|^2 - (R+\del^\gamma)^2) + 1
\]
and taking $\delta$ smaller if needed.
}
Moving any supporting plane to $\w^+$ in $Z$ down and then up until it touches $\u$ from below, we see that
\[
\nab \w^+(Z) \subset \nab \u(Z).
\]
It follows that $\{ (x,\w^+(x),\nab \w^+(x)) : x \in Z \} \subset \dom \Gexp$.
Let $x_0^+ \in Z$, $y_0 := \T_{\w^+}(x_0^+)$, and $v_0^+ := \H(x_0^+,y_0,\w^+(x_0^+))$.
Increase and then decrease $v_0^+$ to $v_0$ so that $\G(\cdot,y_0,v_0)$ touches $\u$ from below at $x_0$.
Recall that $\T_{\w^+}(x_0^+) = y_0$ if and only if $\nab \w^+(x_0^+) = D_x\G(x_0^+,y_0,v_0^+)$.
(See Section~\ref{sec: prelim Gconvex}.)
Hence, from \eqref{eqn: uni convex w plus} and as $\G(x_0^+,y_0,v_0^+) = \w^+(x_0^+)$, we have that
\[
\G(x,y_0,v_0^+) \leq \w^+(x_0^+) + \nab\w^+(x_0^+) \cdot (x-x_0^+) + \frac{\del}{2}|x-x_0^+|^2 \leq \w^+(x) \quad\forall x \in \co[S].
\]
In turn, $x_0 \in Z$.
Indeed, if not, then
\[
\G(x_0,y_0,v_0) = \u(x_0) \geq \w^+(x_0) \geq \G(x_0,y_0,v_0^+),
\]
from which using (G-Mono), it follows that
\[
\w^+(x_0^+) > \u(x_0^+) \geq \G(x_0^+,y_0,v_0) \geq \G(x_0^+,y_0,v_0^+) = \w^+(x_0^+).
\]
Impossible; and we deduce that 
\begin{equation}
\label{eqn: Gexp w plus in Tu}
\T_{\w^+}(Z) \subset \T_\u(Z).
\end{equation}
Now for any $x \in Z$, from \eqref{eqn: bounds on D2w} and taking $\del$ even smaller, we compute that
\[
\begin{split}
D^2 \w^+(x) - D^2_x \G(x,\T_{\w^+}(x),\H(x,\T_{\w^+}(x),\w^+(x)))
\geq  (1 + 3\vep+ \del^{1/2})D^2 \w(x).
\end{split}
\]
And so by \eqref{eqn: prop G close to linear}, we see that
\[
\begin{split}
\frac{\det(D^2 \w^+(x) - D^2_x \G(x,\T_{\w^+}(x),\H(x,\T_{\w^+}(x),\w^+(x))))}{|\det (\E(x,\T_{\w^+}(x),\H(x,\T_{\w^+}(x),\w^+(x))))|}
\geq \frac{(1+ 3\vep+\del^{1/2})^n}{(1+ \del)^n}\frac{\l_1}{\l_2}.
\end{split}
\] 
Moreover, from \eqref{eqn: uni convex w plus}, for any $x,z \in Z$ with $x \neq z$, setting $y := \T_{\w^+}(x)$ and $v := \H(x,y,\w^+(x))$, we determine that
\[
\w^+(z) - \G(z,y,v) = \frac{1}{2}\int_0^1 \big(D^2\w^+(tz + (1-t)x) - D^2_x\G(tz + (1-t)x,y,v)\big)(z-x)\cdot (z-x) \, \d t > 0.
\]
In other words, the function $\G(\cdot,y,v)$ only touches $\w^+$ at $x$, and the map $x \mapsto \T_{\w^+}(x)$ is injective in $Z$.
Therefore, Lemma~\ref{lem: cov Gexp} yields
\[
|\T_{\w^+}(x)| \geq \frac{(1+ 3\vep +\del^{1/2})^n}{(1+\del)^n}\frac{\l_1}{\l_2}|Z| > (1+3\vep)\frac{\l_1}{\l_2}|Z|
\]
if $\del$ is small enough depending only on $R$ and $n$.
On the other hand, since $\u$ is $C^1$ in $S$, the push-forward condition and \eqref{eqn: prop densities close to cst} imply that
\[
|\T_\u(Z)| = \int_Z \frac{f(x)}{g(\T_\u(x))} \, \d x \leq \frac{1+\vep}{1-\vep}\frac{\l_1}{\l_2}|Z|.
\]
Combining these last two inequalities, we find that \eqref{eqn: Gexp w plus in Tu} is impossible unless $Z$ is empty.
That is, $\w^+ \leq \u$ in $S$.

The argument showing that $\u \leq \w^-$ in $S$ is similar to the one just presented, showing that $\u \geq \w^+$ in $S$.
So we only provide a sketch.
Again, suppose, to the contrary, that $W := \{ \u > \w^- \}$ is non-empty.
Now we can find a positive constant $\mu$ so that $\u$ touches $\w^- + \mu$ from below in $S$.
As both $\u$ and $\w^-$ are $C^1$, it follows that $\nab \u = \nab \w^-$ on the set $\{ \u = \w^- + \mu \}$.
Therefore, if $\eta > 0$ is sufficiently small, then the set $W_\eta : = \{ \u > \w^- + \mu - \eta \}$ is non-empty and  $\nab \w^-(W_\eta)$ is contained in a small neighborhood of $\nab \u(W_\eta)$.

Set $\w^-_\eta := \w^- + \mu - \eta$.
Then, using the same barrier as before, we find that $\w^-_\eta \in (-R^2,2)$.
Hence, $\{ (x,\w^-_\eta(x),\nab \w^-_\eta(x)) : x \in W_\eta\} \subset \dom \Gexp$.
Let $\G(\cdot,y_0,v_0^-)$ be the $\G$-support for $\u$ at $x_0^- \in W_\eta$.
Increase and then decrease $v_0^-$ to $v_0$ so that $\G(\cdot,y_0,v_0)$ touches $\w^-_\eta$ from below, and let $x_0$ be the point at which  $\G(\cdot,y_0,v_0)$ touches $\w^-_\eta$ from below.
Notice that $x_0 \in W_\eta$ and $v_0 = \H(x_0,y_0,\w^-_\eta(x_0))$.
Therefore, $D_x\G(x_0,y_0,v_0) = \nab \w_\eta^-(x_0)$; that is, $y_0 \in \T_{\w_\eta^-}(W_\eta)$ or
\[
\T_{\u}(W_\eta) \subset \T_{\w_\eta^-}(W_\eta).
\]
Observe that from \eqref{eqn: bounds on D2w},
\[
(1- 3\vep - \del^{1/2})D^2 \w(x) \leq D^2 \w^-_\eta(x) - D^2_x \G(x,\T_{\w^-_\eta}(x),\H(x,\T_{\w^-_\eta}(x),\w^-_\eta(x)))
\]
if $\del^{3/4} \geq 2K_{R,n}\del$.
Also, taking $\del$ even smaller (so that $\del^{3/4} \geq 4K_{R,n}\del$ or, equivalently, $\del^{1/2} \geq 4K_{R,n}\del^{3/4}$), we find that
\[
D^2 \w^-_\eta(x) - D^2_x \G(x,\T_{\w^-_\eta}(x),\H(x,\T_{\w^-_\eta}(x),\w^-_\eta(x))) 
\leq \bigg(1 - 3\vep - \frac{\del^{1/2}}{4}\bigg)D^2 \w.
\]
Without loss of generality, we assume that $\del^{1/2} \leq 1/4$; whence, $1 - 3\vep - \del^{1/2} \geq 0$.
Hence, by Lemma~\ref{lem: cov Gexp},
\[
|\T_{\w_\eta^-}(W_\eta)| \leq \frac{(1  - 3\vep - \del^{1/2}/4)^n}{(1-\del)^{n}}\frac{\l_2}{\l_2}|W_\eta| < (1-3\vep)\frac{\l_2}{\l_2}|W_\eta|.
\]
Moreover,
\[
|\T_\u(W_\eta)| \geq \frac{1-\vep}{1+\vep}\frac{\l_1}{\l_2}|W_\eta|.
\]
Like before, combining these last two inequalities, we arrive at a contradiction unless $|W_\eta| = 0$, so long as $\del$ is sufficiently small depending on $R$ and dimension. 
\end{proof}

With Proposition~\ref{prop: comparison} in hand, our next proposition is a higher regularity version of Proposition~\ref{prop: C1b ptwise}.

\begin{proposition}
\label{prop: hr ptwise}
In addition to the hypotheses of Proposition~\ref{prop: C1b ptwise}, suppose $f \in C^{0,\alpha}(\C)$, $g \in C^{0,\alpha}(\K)$, and $B_{6} \times (-27,27) \times B_{6} \subset \dom \Gexp$.
There exist positive constants $\del' \leq \del$ and $\eta' \leq \eta$ such that the following holds: if
\[
\|f - 1\|_{L^\infty(\C)} + \|g - 1\|_{L^\infty(\K)} \leq \del',
\]
\begin{equation}
\label{eqn: G small hr assump}
\|\G - x \cdot y + v\|_{C^2(\B_3)} + \|\H - x \cdot y + v\|_{C^2(\B_3)} \leq \del',
\end{equation}
and
\[
\bigg{\|}\u - \frac{1}{2}|x|^2\bigg{\|}_{C^0(B_3)} \leq \eta',
\]
then $\u \in C^{2,\alpha'}(0)$ for some $\alpha' < \alpha$.
\end{proposition}

The proof of Proposition~\ref{prop: hr ptwise} follows arguing exactly as in the proof of \cite[Theorem~5.3]{DF}.
That said, let us make some remarks.
An inspection of the proof of \cite[Theorem~5.3]{DF} reveals that, apart from a comparison-type principle like Proposition~\ref{prop: comparison}, we will need that the sum of the norms in \eqref{eqn: G small hr assump} decays under parabolically quadratic rescalings.
(The remainder of the proof uses classical estimates for the Monge--Amp\`{e}re equation.)
The assumption $\G,\H \in C^{2,\alpha}(\B_3)$ plus \eqref{eqn: extra 1} and \eqref{eqn: extra 2} guarantee this decay.
Finally, since the domain of the $\Gexp$ map includes the product of three open sets at the beginning, the third of which compactly contains $\nab \u(B_{1/3})$, the set inclusions in the hypotheses of Proposition~\ref{prop: comparison} will be satisfied at each stage of the iteration by construction; we are zooming in with the correct rescaling.
So applying Proposition~\ref{prop: hr ptwise} at every point in $B_{1/7}$ and then classical Schauder estimates, we obtain our final theorem.

\begin{theorem}
\label{thm: hr}
In addition to the hypotheses of Theorem~\ref{thm: close implies C1b}, suppose $\G,\H \in C^{k,\alpha}(\B_4)$, $f \in C^{k,\alpha}(\C)$, $g \in C^{k,\alpha}(\K)$, for some $k \geq 0$ and $\alpha \in (0,1)$, and $B_{8} \times (-64,64) \times B_{8} \subset \dom \Gexp$.
There exist positive constants $\del_1 \leq \del_0$ and $\eta_1 \leq \eta_0$ such that the following holds: if
\[
\|f - \1_{\C}\|_{L^\infty(B_4)} + \|g - \1_{\K}\|_{L^\infty(B_4)} \leq \del_1,
\]
\[
\|\G - x \cdot y + v\|_{C^2(\B_4)} + \|\H - x \cdot y + v\|_{C^2(\B_4)} \leq \del_1,
\]
and
\[
\bigg{\|}\u - \frac{1}{2}|x|^2\bigg{\|}_{C^0(B_4)} \leq \eta_1,
\]
then $\u \in C^{k+2,\alpha}(B_{1/8})$.
\end{theorem}

\noindent{\bf Acknowledgments.}
I would like to thank Jun Kitagawa for a careful reading of preliminary versions of this paper and discussions and Nestor Guillen for a helpful conversation.
Also, I am grateful to both Connor Mooney and Alessio Figalli for their support, encouragement, and valuable conversations.
Finally, I wish to thank the referee for his or her detailed reading and comments.



\begin{thebibliography}{99}

\bibitem{B}
	\newblock G. A. Beer, 
	\newblock {\it The Hausdorff metric and convergence in measure}, 
	\newblock Michigan Math. J. {\bf 21} (1974), 63-64.

\bibitem{C} L. A. Caffarelli,
	\newblock {\em The regularity of mappings with a convex potential},
	\newblock J. Amer. Math. Soc. {\bf 5}(1) (1992), 99-104.

\bibitem{CF15} S. Chen and A. Figalli,
	\newblock {\em Boundary $\vep$-regularity in optimal transportation}, 
	\newblock  Adv. Math. {\bf 273} (2015), 540-567.	

\bibitem{CF17}
	\newblock S. Chen and A. Figalli, 
	\newblock {\it Partial $W^{2,p}$ regularity for optimal transport maps}, 
	\newblock J. Funct. Anal. {\bf 272} (2017), no. 11, 4588-4605.

\bibitem{DF}
	\newblock G. De Philippis and A. Figalli, 
	\newblock {\it Partial regularity for optimal transport maps}, 
	\newblock Publ. Math. Inst. Hautes \'{E}tudes Sci. (2015), no. 121, 81-112.

\bibitem{EG}
	\newblock L. C. Evans and R. F. Gariepy, 
	\newblock {``Measure Theory and Fine Properties of Functions''}, 
	\newblock CRC Press Inc. Boca Raton, 1992.	

\bibitem{FF}
	\newblock A. Fathi and A. Figalli, 
	\newblock {\it Optimal transportation on non-compact manifolds}, 
	\newblock Israel J. Math. {\bf 175} (2010), no. 1, 1-59.

\bibitem{F}
	\newblock A. Figalli, 
	\newblock {\em Regularity properties of optimal maps between nonconvex domains in the plane},  
	\newblock Comm. Partial Differential Equations {\bf 35} (2010), no. 3, 465-479.		

\bibitem{FK}
	\newblock A. Figalli and Y.-H. Kim, 
	\newblock {\em Partial regularity of Brenier solutions of the Monge-Amp\`{e}re equation}, 
	\newblock Discrete Contin. Dyn. Syst. {\bf 28} (2010), no. 2, 559-565.

\bibitem{GO}
	\newblock M. Goldman and F. Otto, 
	\newblock {\it A variational proof of partial regularity for optimal transportation maps},
	\newblock arXiv:1704.05339.			

\bibitem{GK}
	\newblock N. Guillen and J. Kitagawa, 
	\newblock {\it Pointwise estimates and regularity in geometric optics and other generated Jacobian equations}, 
	\newblock Comm. Pure Appl. Math. {\bf 70} (2017), no. 6, 1146-1220.

\bibitem{J}
	\newblock Y. Jhaveri, 
	\newblock {\it On the (in)stability of the identity map in optimal transportation},
	\newblock arXiv:1710.03708.

\bibitem{KW}
	\newblock A. Karakhanyan and X.-J. Wang, 
	\newblock {\it On the reflector shape design}, 
	\newblock J. Differential Geom. {\bf 84} (2010), no. 3, 561-610.	

\bibitem{MTW}
	\newblock X.-N. Ma, N. S. Trudinger, and X.-J. Wang, 
	\newblock {\it Regularity of potential functions of the optimal transportation problem}, 
	\newblock Arch. Ration. Mech. Anal. {\bf 177} (2005), no. 2, 151-183.	

\bibitem{T}
	\newblock N. S. Trudinger, 
	\newblock {\it On the local theory of prescribed Jacobian equations}, 
	\newblock Discrete Contin. Dyn. Syst. {\bf 34} (2014), no. 4, 1663-1681.

\bibitem{V} 
	\newblock C. Villani, 
    \newblock {``Optimal Transport, Old and New''}, 
    \newblock Grundlehren des mathematischen Wissenschaften [Fundamental Principles of Mathematical Sciences], Vol. 338, Springer-Verlag Berlin Heidelberg, 2009.			

\end{thebibliography}
\end{document}